\documentclass[12pt]{article}
\ifdefined\usebigfont

\usepackage{times}
\usepackage[fontsize=13pt]{scrextend}
\usepackage[left=1.56in,right=1.56in,top=1.74in,bottom=1.74in]{geometry}
\else
\usepackage[margin=1in]{geometry}
\fi

\usepackage{color}
\usepackage{amssymb,amsthm}
\usepackage{amsmath}
\usepackage{enumerate}
\usepackage{hyperref}
\usepackage{cite}
\usepackage{mathtools}

\numberwithin{equation}{section}

\newcommand{\inclu}[0] {\ar@{^{(}->}}

\newcommand{\spann}{\text{span}}

\newcommand{\gph}{{\rm gph}\,}

\newcommand{\dist}{{\rm dist}}
\newcommand{\R}{\mathbb{R}}

\newcommand{\EE}{\mathbb{E}}

\newcommand{\Null}{\mathrm{Null}}

\newcommand{\RR}{\mathbb{R}}

\newcommand{\conv}{\text{conv}\,}
\newcommand{\lf}{\operatornamewithlimits{liminf}}

\newcommand{\cF}{\mathcal{F}}

\newcommand{\cl}{\mathrm{cl}\,}

\newcommand{\cX}{\mathcal{X}}

\newcommand{\argmin}{\operatornamewithlimits{argmin}}
\newcommand{\lims}{\operatornamewithlimits{limsup}}

\newcommand{\proj}{{\rm proj}}

\newcommand{\NN}{\mathbb{N}}

\newtheorem{thm}{Theorem}[section]
\newtheorem{definition}[thm]{Definition}
\newtheorem{proposition}[thm]{Proposition}
\newtheorem{lem}[thm]{Lemma}
\newtheorem{cor}[thm]{Corollary}

\newtheorem{assumption}{Assumption}

\theoremstyle{remark}
\newtheorem{claim}{Claim}

\usepackage{mathtools}
\DeclarePairedDelimiter{\dotp}{\langle}{\rangle}
\usepackage[boxruled]{algorithm2e}

\begin{document}
	
	\title{Stochastic subgradient method converges\\ on tame functions}
	
	
	\author{Damek Davis\thanks{School of Operations Research and Information Engineering, Cornell University,
Ithaca, NY 14850, USA;
\texttt{people.orie.cornell.edu/dsd95/}.}\and 
Dmitriy Drusvyatskiy\thanks{Department of Mathematics, University of Washington, Seattle, WA 98195; \texttt{www.math.washington.edu/{\raise.17ex\hbox{$\scriptstyle\sim$}}ddrusv}. Research of Drusvyatskiy was supported by the AFOSR YIP award FA9550-15-1-0237 and by the NSF DMS   1651851 and CCF 1740551 awards.}			
\and  Sham Kakade\thanks{Departments of Statistics and Computer Science, University of Washington, Seattle, WA 98195; \texttt{homes.cs.washington.edu/{\raise.17ex\hbox{$\scriptstyle\sim$}}sham/}. Sham Kakade acknowledges funding from the Washington Research Foundation Fund for Innovation in Data-Intensive Discovery and the  NSF CCF 1740551 award.} \and Jason D. Lee\thanks{Data Science and Operations Department, Marshall School of Business, University of Southern California, Los Angeles, CA 90089; \texttt{www-bcf.usc.edu/{\raise.17ex\hbox{$\scriptstyle\sim$}}lee715}. JDL acknowledges funding from the ARO MURI Award W911NF-11-1-0303. }
			}

	\date{}
	\maketitle

\begin{abstract}
 This work considers the question: what convergence guarantees does the stochastic subgradient method have in the absence of smoothness and convexity? 
We prove that the stochastic subgradient method, on any semialgebraic locally Lipschitz function, produces limit points that are all first-order stationary. More generally, our result applies to any function with a Whitney stratifiable graph. In particular, this work endows the stochastic subgradient method, and its proximal extension, with rigorous convergence guarantees for a wide class of problems arising in data science---including all popular deep learning architectures.

\end{abstract}
\section{Introduction}
In this work, we study the long term behavior of the stochastic subgradient method on nonsmooth and nonconvex functions. Setting the stage, consider the optimization problem
\begin{equation*} 
\min_{x\in \RR^d}~ f(x), 
\end{equation*}
where $f \colon \RR^d \rightarrow \RR$ is a locally Lipschitz continuous function. The stochastic subgradient method simply iterates the steps
\begin{equation}\label{eqn:stock_subgrad}
x_{k+1}=x_k-\alpha_k \Big(y_k +  \xi_k\Big)\qquad \textrm{with}\qquad y_k\in \partial f(x_k).
\end{equation}
Here $\partial f(x)$ denotes the Clarke subdifferential \cite{ClarkeTAMS}. Informally, the set $\partial f(x)$ is the convex hull of limits of gradients at nearby differentiable points. 
 In classical circumstances, the subdifferential reduces to more familiar objects. Namely, when $f$ is $C^1$-smooth at $x$, the subdifferential $\partial f(x)$ consists only of the gradient $\nabla f(x)$, while for convex functions, it reduces to the subdifferential in the sense of convex analysis. The positive sequence $\{\alpha_k\}_{k\geq 0}$ is user specified, and it controls the step-sizes of the algorithm. As is typical for stochastic subgradient methods, we will assume that this sequence is square summable but not summable, meaning $\sum_{k}\alpha_k=\infty$ and $\sum_{k}\alpha_k^2<\infty$.
Finally, the stochasticity is modeled by the random (noise) sequence $\{\xi_k\}_{k\geq 1}$. We make the standard assumption that conditioned on the past, each random variable $\xi_k$ has mean zero and its second moment grows at a controlled rate. 

 Though variants of the stochastic subgradient method \eqref{eqn:stock_subgrad} date back to Robbins-Monro's pioneering 1951 work~\cite{stoch_rob_monr}, their convergence behavior is still largely not understood in nonsmooth and nonconvex settings. In particular, the following question remains open.
\begin{quote}
\centering Does the (stochastic) subgradient method have any convergence guarantees on locally Lipschitz functions, which may be neither smooth nor convex?
\end{quote}
That this question remains unanswered is somewhat concerning as the
stochastic subgradient method forms a core numerical subroutine for
several widely used solvers, including Google's
TensorFlow~\cite{tensorflow2015-whitepaper} and the open source
PyTorch~\cite{paszke2017automatic} library. 

Convergence behavior of~\eqref{eqn:stock_subgrad} is well understood when applied to convex, smooth, and more generally, weakly convex problems. In these three cases, almost surely, every limit point $x^*$ of the iterate sequence is first-order critical \cite{Nurminskii1974}, meaning $0\in \partial f(x^*)$. Moreover, rates of convergence in terms of natural optimality/stationarity measures are available.
In summary, the rates are $\EE\left[ f(x_k) - \inf f\right]=O(k^{-1/2})$, $\EE\left[\|\nabla f(x_k)\|\right]=O(k^{-1/4})$, and $\EE\left[\|\nabla f_{1/(2\rho)}(x_k)\|\right]=O(k^{-1/4})$, for functions that are convex~\cite{complexity}, smooth~\cite{ghad}, and $\rho$-weakly convex~\cite{weak_conv_papr,model_paper}, respectively. 
In particular, the convergence guarantee above for $\rho$-weakly convex functions appeared only recently in \cite{weak_conv_papr,model_paper}, with the Moreau envelope $f_{1/(2\rho)}$ playing a central role.

Though widely applicable, these previous results on the convergence of
the stochastic subgradient method do not apply to even relatively
simple non-pathological functions, such as 
$f(x,y) = (|x|-|y|)^2$ and $f(x) = (1-\max\{x,0\})^2$. 
It is not only toy examples, however, that lack
convergence guarantees, but the entire class of deep neural networks
with nonsmooth activation functions (e.g., ReLU). Since such networks
are routinely trained in practice, it is worthwhile to understand if
indeed the iterates $x_k$ tend to a meaningful limit.

In this paper, we provide a positive answer to this question for a wide class of locally Lipschitz functions; indeed, the function class we consider is virtually exhaustive in data scientific contexts (see Corollary~\ref{cor:deep} for consequences in  deep learning). Aside from mild technical conditions, the only meaningful assumption we make is that $f$ strictly decreases along any trajectory $x(\cdot)$ of the differential inclusion $\dot x(t) \in -\partial f(x(t))$ emanating from a noncritical point. Under this assumption, a standard Lyapunov-type argument shows that every limit point of the stochastic subgradient method is critical for $f$, almost surely. Techniques of this type can be found for example in the  monograph of Kushner-Yin \cite[Theorem 5.2.1]{kush_yin} and the landmark papers of Bena\"\i m-Hofbauer-Sorin \cite{BHS,BHS2}. Here, we provide a self-contained treatment, which facilitates direct extensions to ``proximal'' variants of the stochastic subgradient method.\footnote{Concurrent to this work, the independent preprint \cite{indep_paper} also provides convergence guarantees for the stochastic projected subgradient method, under the assumption that the objective function is ``subdifferentially regular'' and the constraint set is convex. Subdifferential regularity rules out functions with downward kinks and cusps, such as deep networks with the Relu($\cdot$)  activation functions. Besides subsuming the subdifferentially regular case, the results of the current paper apply to the broad class of Whitney stratifiable functions, which includes all popular deep network architectures.} In particular, our analysis follows closely the recent work of Duchi-Ruan \cite[Section 3.4.1]{duchi_ruan} on convex composite minimization.

The main question that remains therefore is which functions  decrease along the continuous subgradient curves. Let us look for inspiration at convex functions, which are well-known to satisfy this property \cite{Brezis,bruck}. Indeed, if $f$ is convex and $x\colon [0,\infty)\to\R$ is any absolutely continuous curve, then the ``chain rule'' holds:
\begin{equation}\label{eqn:main_chain_rule}
\tfrac{d}{dt}(f\circ x)=\langle \partial f(x),\dot{x} \rangle \qquad \textrm{for a.e. }t\geq 0.
\end{equation}
An elementary linear algebraic argument then shows that if $x$ satisfies $\dot x(t)\in -\partial f(x(t))$ a.e.,
then automatically $-\dot x(t)$ is the minimal norm element of $\partial f(x(t))$. Therefore, integrating \eqref{eqn:main_chain_rule} yields the desired descent guarantee
\begin{equation}\label{eqn:descent}
f(x(0))-f(x(t))=\int^t_{\tau=0} \dist^2(0;\partial f(x(\tau))) \qquad \textrm{for all }t\geq 0.
\end{equation}
Evidently, exactly the same argument yields the chain rule \eqref{eqn:main_chain_rule} for {\em subdifferentially regular} functions. These are the functions $f$ such that each subgradient $v\in \partial f(x)$  defines a linear lower-estimator of $f$ up to first-order; see for example \cite[Section 2.4]{clarke2008nonsmooth} or \cite[Definition 7.25]{RW98}. Nonetheless, subdifferentially regular functions preclude ``downwards cusps'', and therefore still do not capture such simple examples as $f(x) = (1-\max\{x,0\})^2$. It is worthwhile to mention that one can not expect \eqref{eqn:descent} to always hold. Indeed, there are pathological locally Lipschitz functions $f$ that do not satisfy \eqref{eqn:descent}; one example is the  univariate 1-Lipschitz  function whose Clarke subdifferential is the unit interval at every point \cite{theory_subgrad,gen_lip}.

In this work, we isolate a different structural property on the function $f$, which guarantees the validity of  \eqref{eqn:main_chain_rule} and therefore of the descent condition \eqref{eqn:descent}. We will assume that the graph of the function $f$ admits a partition into finitely many smooth manifolds, which fit together in a regular pattern. Formally, we require the graph of $f$ to admit a so-called  Whitney stratification, and we will call such functions Whitney stratifiable.
Whitney stratifications have already figured prominently in optimization, beginning with the seminal work \cite{bolte2007clarke}.
An important subclass of Whitney stratifiable functions consists of semi-algebraic functions \cite{loja_semi} -- meaning those whose graphs can be written as a finite union of sets each defined by finitely many polynomial  inequalities. Semialgebraicity is preserved under all the typical functional operations in optimization (e.g. sums, compositions, inf-projections)
and  therefore semi-algebraic functions are usually easy to recognize. More generally still, ``semianalytic'' functions \cite{loja_semi} and those that are ``definable in an o-minimal structure'' are Whitney stratifiable \cite{Dries-Miller96}. The latter function class, in particular, shares all the robustness and analytic properties of semi-algebraic functions, while encompassing many more examples. Case in point, Wilkie \cite{exp_omin} famously showed that there is an o-minimal structure that contains both the exponential $x\mapsto e^x$ and all semi-algebraic functions.\footnote{The term ``tame'' used in the title has a technical meaning. Tame sets are those whose intersection with any ball is definable in some o-minimal structure. The manuscript \cite{Ioffe2008} provides a nice exposition on the role of tame sets and functions in optimization.}

The key observation for us, which originates in \cite[Section 5.1]{DIL}, is that any locally Lipschitz Whitney stratifiable function necessarily satisfies the chain rule \eqref{eqn:main_chain_rule} along any absolutely continuous curve. Consequently, the descent guarantee \eqref{eqn:descent} holds along any subgradient trajectory, and our convergence guarantees for the stochastic subgradient method become applicable.   
 Since the composition of two definable functions is definable, it follows immediately from Wilkie's o-minimal structure that nonsmooth deep neural networks built from definable pieces---such as quadratics $t^2$, hinge losses $\max\{0,t\}$, and log-exp $\log(1+e^t)$ functions---are themselves definable. Hence, the results of this paper endow stochastic subgradient methods, applied to definable deep networks, with rigorous convergence guarantees.

Validity of the chain rule  \eqref{eqn:main_chain_rule} for Whitney stratifiable functions is not new. It was already proved in \cite[Section 5.1]{DIL} for semi-algebraic functions, though identical arguments hold more broadly for Whitney stratifiable functions. These results, however, are somewhat hidden in the paper \cite{DIL}, which is possibly why they have thus far been  underutilized. In this manuscript, we provide a self-contained review of the material from \cite[Section 5.1]{DIL}, highlighting only the most essential ingredients and streamlining some of the arguments.

Though the discussion above is for unconstrained problems, the techniques we develop apply much more broadly to constrained problems of the form
$$\min_{x\in \cX}~ f(x)+g(x).$$
Here $f$ and $g$ are locally-Lipschitz continuous functions and $\cX$ is an arbitrary closed set. The popular {\em proximal stochastic subgradient method} simply iterates the steps
\begin{equation} 
\left\{\begin{aligned}
&\textrm{Sample an estimator }\zeta_k \textrm{ of } \partial f(x_k)\\
&\textrm{Select }x_{k+1} \in \argmin_{x\in \cX} ~\left\{\langle \zeta_k,x\rangle+g(x) +\tfrac{1}{2\alpha_k}\|x-x_k\|^2\right\}
\end{aligned}\right\}.
\end{equation}
 Combining our techniques with those in \cite{duchi_ruan} quickly yields subsequential convergence guarantees for this algorithm. Note that we impose no convexity assumptions on $f$, $g$, or $\cX$.

 \indent The outline of this paper is as follows. In Section~\ref{sec:preliminaries}, we fix the notation for the rest of the manuscript. Section~\ref{differential_inclusions} provides a self-contained treatment of asymptotic  consistency for discrete approximations of differential inclusions. In Section~\ref{section:subgradient}, we specialize the results of the previous section to the stochastic subgradient method.
  Finally, in Section~\ref{sec:verifying}, we verify the sufficient conditions for subsequential convergence   for a broad class of locally Lipschitz functions, including those that are subdifferentially regular and Whitney stratifiable.  In particular, we specialize our results to deep learning settings in Corollary~\ref{cor:deep}. In the final Section~\ref{sec:prox_ext}, we extend the results of the previous sections to the proximal setting.

\section{Preliminaries}\label{sec:preliminaries}
Throughout, we will mostly use standard notation on differential inclusions, as set out for example in the monographs of Borkar~\cite{Borkar}, Clarke-Ledyaev-Stern-Wolenski \cite{clarke2008nonsmooth}, and Smirnov \cite{intro_diff_inc}. We will always equip the Euclidean space $\R^d$ with an inner product $\langle\cdot,\cdot \rangle$ and the induced norm $\|x\|:=\sqrt{\langle x,x\rangle}$. The distance of a point $x$ to a set $Q\subset\R^d$ will be written as $\dist(x;Q):=\min_{y\in Q} \|y-x\|$. The indicator function of $Q$, denoted $\delta_Q$, is defined to be zero on $Q$ and $+\infty$ off it.
The symbol $\mathcal{B}$ will denote the closed unit ball in $\R^d$,  while $\mathcal{B}_{\varepsilon}(x)$ will stand for the closed ball of radius of $\varepsilon>0$ around $x$. We will use $\R_+$ to denote the set of nonnegative real numbers.

\subsection{Absolutely continuous curves}
Any continuous function $x\colon \R_+\to\R^d$ is called a curve in $\R^d$. All curves in $\R^d$ comprise the set $\mathcal{C}(\R_+,\R^d)$. We will say that a sequence of function $f_k$ converges to $f$ in $\mathcal{C}(\R_+,\R^d)$ if $f_k$ converge to $f$ uniformly on compact intervals, that is, for all $T>0$, we have
 $$\lim_{k\to\infty} \sup_{t\in [0,T]}\|f_k(t)-f(t)\|=0.$$
Recall that a curve $x\colon\R_+\to\R^d$ is absolutely continuous if there exists a map $y\colon\R_+\to\R^d$ that is integrable on any compact interval and satisfies $$x(t)=x(0)+\int_{0}^t y(\tau) \,d\tau\qquad\textrm{for all }t\geq 0.$$ Moreover, if this is the case, then equality $y(t)=\dot x(t)$ holds for a.e. $t\geq 0$. Henceforth, for brevity, we will call absolutely continuous curves {\em arcs}. We will often use the  observation that if $f\colon\R^d\to\R$ is locally Lipschitz continuous and $x$ is an arc, then the composition $f\circ x$ is absolutely continuous.

\subsection{Set-valued maps and the Clarke subdifferential}
A set-valued map $G\colon \cX\rightrightarrows\R^m$ is a mapping from a set $\cX\subseteq\R^d$ to the powerset of $\R^m$. 
Thus $G(x)$ is a subset of $\R^m$, for each $x\in \cX$. We will use the notation $$G^{-1}(v):=\{x\in \cX:v\in G(x)\}$$ for the preimage of a vector $v\in \R^m$. 
The map $G$ is  {\em outer-semicontinuous} at a point $x\in \cX$ if for any sequences $x_i\xrightarrow[]{\cX} x$ and $v_i\in G(x_i)$  converging to some vector $v\in\R^m$, the inclusion $v\in G(x)$ holds.

The most important set-valued map for our work will be the generalized derivative in the sense of Clarke \cite{ClarkeTAMS} -- a notion we now review. 
Consider a locally Lipschitz continuous function $f\colon\R^d\to\R$.
The well-known Rademacher's theorem guarantees that $f$ is differentiable almost everywhere. Taking this into account, the {\em Clarke subdifferential} of $f$ at any point $x$ is the set \cite[Theorem 8.1]{clarke2008nonsmooth}
$$\partial f(x):=\conv\left\{\lim_{i\to\infty} \nabla f(x_i): x_i\xrightarrow[]{\Omega} x\right\},$$
where $\Omega$ is any full-measure subset of $\R^d$ such that $f$ is differentiable at each of its points. 
It is standard that the map $x\mapsto \partial f(x)$ is outer-semicontinuous and its images $\partial f(x)$ are nonempty, compact,  convex sets for each $x\in \R^d$; see for example \cite[Proposition 1.5 (a,e)]{clarke2008nonsmooth}.

Analogously to the smooth setting, a point $x\in\R^d$ is called {\em (Clarke) critical} for $f$ whenever the inclusion $0\in \partial f(x)$ holds. Equivalently, these are the points at which the Clarke directional derivative is nonnegative in every direction \cite[Section 2.1]{clarke2008nonsmooth}. A real number $r\in\R$ is called a {\em critical value} of $f$ if there exists a critical point $x$ satisfying $r=f(x)$.

\section{Differential inclusions and discrete approximations}\label{differential_inclusions}
In this section, we discuss the asymptotic behavior of discrete approximations of differential inclusions. All the elements of the analysis we present, in varying generality, can be found in the works of Bena\"\i m-Hofbauer-Sorin \cite{BHS,BHS2}, Borkar~\cite{Borkar}, and Duchi-Ruan \cite{duchi_ruan}. Out of these, we most closely follow the work of Duchi-Ruan \cite{duchi_ruan}.

\subsection{Functional convergence of discrete approximations}
Let $\cX$ be a closed set and let $G\colon \cX\rightrightarrows\R^d$ be a set-valued map. Then an arc $x\colon\R_+\to\R^d$ is called a {\em trajectory} of $G$ if it satisfies the differential inclusion
\begin{equation}\label{eqn:diff_incl}
\dot x(t) \in G(x(t)) \qquad \textrm{for a.e.}~ t \geq 0.
\end{equation}
Notice that the image of any arc $x$ is automatically contained in $\cX$, since arcs are continuous and  $\cX$ is closed. In this work, we will primarily focus on iterative algorithms that aim to asymptotically track a trajectory of the differential inclusion \eqref{eqn:diff_incl} using a noisy discretization with vanishing step-sizes. Though our discussion allows for an arbitrary set-valued map $G$, the reader should keep in mind that the most important example for us will be  $G=-\partial f$, where $f$ is a locally Lipschitz function.

Throughout, we will consider the following iteration sequence: 
\begin{equation}\label{eqn:Euler_rec}
x_{k+1}=x_k+\alpha_k (y_k +  \xi_k).
\end{equation}
Here $\alpha_k>0$ is a sequence of step-sizes, $y_k$ should be thought of as an approximate evaluation of $G$ at some point near $x_k$, and  $\xi_k$ is a  sequence of ``errors".

Our immediate goal is to isolate reasonable conditions, under which 
the sequence $\{x_k\}$ asymptotically tracks a trajectory of the differential inclusion \eqref{eqn:diff_incl}. Following the work of Duchi-Ruan \cite{duchi_ruan} on stochastic approximation, we stipulate the following assumptions. 
\begin{assumption}[Standing assumptions]\label{ass:duchi_assumption}
~
\begin{enumerate}
\item \label{item:iterates_in_Q} All limit points of $\{x_{k}\}$ lie in $\cX$.
\item \label{item:gradients_bounded} The iterates are bounded, i.e., $\sup_{k \geq 1} \|x_k\| < \infty$ and $\sup_{k \geq 1} \|y_k\| < \infty$.
\item \label{item:steps_duchi}The sequence $\{\alpha_k\}$ is nonnegative, square summable, but not summable: $$\alpha_k\geq 0, \qquad\sum_{k=1}^{\infty}\alpha_k=\infty,\qquad \textrm{and}\qquad \sum_{k=1}^{\infty}\alpha_k^2<\infty.$$
\item \label{item:noise}The weighted noise sequence is convergent: $\sum_{k = 1}^n \alpha_k \xi_{k} \rightarrow v$ for some $v$ as $k \rightarrow \infty$.
\item \label{item:average_OSC} 
For any unbounded increasing sequence $\{k_j\}\subset \NN$ such that $x_{k_j}$ converges to some point $\bar x$, it holds:
\begin{align*}
\lim_{n \rightarrow \infty}\dist\left(\frac{1}{n} \sum_{j = 1}^n y_{k_j}, G(\bar x)\right) = 0.
\end{align*}
\end{enumerate}
\end{assumption}

Some comments are in order. Conditions~\ref{item:iterates_in_Q}, \ref{item:gradients_bounded}, and \ref{item:steps_duchi}  are in some sense minimal, though the boundedness condition must be checked for each particular algorithm.  Condition \ref{item:noise} guarantees that the noise sequence $\xi_k$ does not grow too quickly relative to the rate at which $\alpha_k$ decrease. The key Condition \ref{item:average_OSC} summarizes the way in which the values $y_k$ are approximate evaluations of $G$, up to convexification.

 To formalize the idea of asymptotic approximation, let us define the time points $t_0=0$ and  $t_m=\sum_{k=1}^{m-1} \alpha_k$, for $m\geq 1$. Let $x(\cdot)$ now be the linear interpolation of the discrete path:
\begin{equation}\label{eqn:desr_path}
x(t):=x_k+\frac{t-t_k}{t_{k+1}-t_k}(x_{k+1}-x_k)\qquad\textrm{ for }t\in [t_k,t_{k+1}).
\end{equation}
For each $\tau\geq 0$, define 
 the time-shifted curve $x^{\tau}(\cdot)=x(\tau+\cdot)$. 

The following result of Duchi-Ruan~\cite[Theorem 2]{duchi_ruan} shows that under the above conditions, for any sequence $\tau_k\to\infty$, the shifted curves $\{x^{\tau_k}\}$ subsequentially converge in $\mathcal{C}(\R_+,\R^d)$ to a trajectory of \eqref{eqn:diff_incl}. Results of this type under more stringent assumptions, and with similar arguments, have previously appeared for example in Bena\"\i m-Hofbauer-Sorin \cite{BHS,BHS2} and Borkar~\cite{Borkar}. 

\begin{thm}[Functional approximation]\label{thm:duchi_arzela}
Suppose that Assumption~\ref{ass:duchi_assumption} holds. Then for any sequence $\{\tau_k\}_{k=1}^\infty  \subseteq \RR_+$, the set of functions $\{x^{\tau_k}(\cdot)\}$ is relatively compact in $\mathcal{C}(\R_+,\R^d)$. If in addition $\tau_k \rightarrow \infty$ as $k \rightarrow \infty$, all limit points $z(\cdot)$ of $\{x^{\tau_k}(\cdot)\}$ in $\mathcal{C}(\R_+,\R^d)$ are 
trajectories of the differential inclusion \eqref{eqn:diff_incl}.
\end{thm}

\subsection{Subsequential convergence to equilibrium points}
A primary application of the discrete process \eqref{eqn:Euler_rec} is to solve the inclusion
\begin{equation}\label{eqn:incl_solve}
0\in G(z).
\end{equation}
Indeed, one can consider the points satisfying \eqref{eqn:incl_solve}
as equilibrium (constant) trajectories of the differential inclusion \eqref{eqn:diff_incl}.
Ideally, one would like to find conditions guaranteeing that every limit point $\bar x$ of the sequence $\{x_k\}$, produced by the recursion \eqref{eqn:Euler_rec}, satisfies the desired inclusion \eqref{eqn:incl_solve}.
Making such a leap rigorous typically relies on combining the asymptotic convergence guarantee of Theorem~\ref{thm:duchi_arzela} with existence of a Lyapunov-like function $\varphi(\cdot)$ for the continuous dynamics; see e.g. \cite{BHS,BHS2}. Let us therefore introduce the following assumption.

\begin{assumption}[Lyapunov condition]\label{ass:b_general}
There exists a continuous function $\varphi \colon \RR^d \rightarrow \RR$, which is bounded from  below, and such that the following two properties hold.
\begin{enumerate}
	\item\label{it:dense0} {\bf (Weak Sard)} For a dense set of values $r\in \R$, the intersection $\varphi^{-1}(r)\cap G^{-1}(0)$ is empty.
	\item\label{ass:b2} {\bf (Descent)}	Whenever $z\colon\R_+\to\R^d$ is a trajectory of the differential inclusion \eqref{eqn:diff_incl} and $0 \notin G(z(0))$, there exists a real $T>0$ satisfying 
$$\varphi(z(T))<\sup_{t\in [0,T]} \varphi(z(t))\leq \varphi(z(0)).$$ 
\end{enumerate}	
\end{assumption}

The weak Sard property is reminiscent of the celebrated Sard's theorem in real analysis. Indeed, consider the classical setting $G=-\nabla f$ for a smooth function $f\colon\R^d\to \R$. Then the weak Sard property stipulates that the set of noncritical values of $f$ is dense in $\R$. By Sard's theorem, this is indeed the case, as long as $f$ is $C^d$ smooth. Indeed, Sard's theorem guarantees
the much stronger property that the set of noncritical values has full measure. We will comment more on the weak Sard property in Section~\ref{section:subgradient}, once we shift focus to optimization problems. The descent property, says that $\varphi$ eventually strictly decreases along the trajectories of the differential inclusion $\dot z \in G(z)$ emanating from any  non-equilibrium point. This Lyapunov-type condition is standard in the literature and we will  verify that it holds for a large class of optimization problems in Section~\ref{sec:verifying}.

As we have alluded to above, the following theorem shows that under Assumptions \ref{ass:duchi_assumption} and \ref{ass:b_general}, every limit point $\bar x$  of $\{x_k\}$ indeed satisfies the inclusion $0\in G(\bar x)$. We were unable to find this result stated and proved in this generality. Therefore, we record a complete proof in Section~\ref{sec:proof_main_thm}. The idea of the proof is of course not new, and can already be seen  for example in \cite{BHS,duchi_ruan,kush_yin}.  Upon first reading, the reader can safely skip to Section~\ref{section:subgradient}.

\begin{thm}\label{thm:main_cray_thm_general}
	Suppose that Assumptions \ref{ass:duchi_assumption} and \ref{ass:b_general} hold. Then every limit point of $\{x_{k}\}_{k\geq 1}$ lies in $G^{-1}(0)$ and the function values $\{\varphi(x_k)\}_{k\geq 1}$ converge.
\end{thm}

\subsection{Proof of Theorem~\ref{thm:main_cray_thm_general}}\label{sec:proof_main_thm}
In this section, we will prove Theorem~\ref{thm:main_cray_thm_general}. The argument we present is rooted in the ``non-escape argument'' for ODEs, using $\varphi$ as a Lyapunov function for the continuous dynamics. In particular, the proof we present is in the same spirit as that in \cite[Theorem 5.2.1]{kush_yin} and \cite[Section 3.4.1]{duchi_ruan}.

Henceforth, we will suppose that 
Assumptions~\ref{ass:duchi_assumption} and~\ref{ass:b_general} hold. We first collect two  elementary lemmas.

\begin{lem}\label{lem:step_size}
The equality $\lim_{k\to\infty}\|x_{k+1}-x_k\|=0$ holds.
\end{lem}
\begin{proof}
From the recurrence \eqref{eqn:Euler_rec}, we have 
 $\|x_{k+1}-x_{k}\|
 \leq \alpha_{k}\|y_{k}\|+\alpha_{k}\|\xi_k\|.$ Assumption~\ref{ass:duchi_assumption} guarantees $\alpha_k\to 0$ and $\{y_k\}$ are bounded, and therefore $\alpha_{k}\|y_{k}\|\to 0$. Moreover, since the sequence 
 $\sum_{k = 1}^n \alpha_k \xi_{k}$ is convergent, we deduce  $\alpha_{k}\|\xi_k\|\to 0$. The result follows.
\end{proof}

\begin{lem}\label{lem_liminf}
Equalities hold: 
\begin{equation}\label{eqn:lfls}
\lf_{t\to\infty} \varphi(x(t))=\lf_{k\to\infty} \varphi(x_k)\qquad \textrm{ and }\qquad \lims_{t\to\infty} \varphi(x(t))=\lims_{k\to\infty}\varphi(x_k).
\end{equation}
\end{lem}
\begin{proof}
	Clearly, the inequalities $\leq$ and $\geq$ hold in \eqref{eqn:lfls}, respectively. We will argue that the reverse inequalities are valid.
	To this end, let $\tau_i\to\infty$ be an arbitrary sequence
	with $x(\tau_i)$ converging to some point $x^*$ as $i\to\infty$.

For each index $i$, define the breakpoint $k_i=\max\{k\in\mathbb{N}: t_k\leq \tau_i\}$. Then by the triangle inequality, we have
	\begin{align*}
	\|x_{k_i}-x^*\|&\leq \|x_{k_i}-x(\tau_i)\|+\|x(\tau_i)-x^*\|
	\leq \|x_{k_i}-x_{k_i+1}\|+\|x(\tau_i)-x^*\|
	\end{align*}
	Lemma~\ref{lem:step_size} implies that the right-hand-side tends to zero, and hence $x_{k_i}\to x^*$. Continuity of $\varphi$ then directly yields the guarantee $\varphi(x_{k_i}) \rightarrow \varphi(x^*)$. 
	
 In particular,	we may take $\tau_i\to \infty$ to be a sequence realizing $\lf_{t\to\infty} \varphi(x(t))$. Since the curve $x(\cdot)$ is bounded, we may suppose that up to taking a subsequence,  $x(\tau_i)$ converges to some point $x^*$. We therefore deduce $$\lf_{k\to\infty} \varphi(x_k)\leq \lim_{i\to\infty}\varphi(x_{k_i}) = \varphi(x^*)=\lf_{t\to\infty} \varphi(x(t)),$$ thereby establishing the first equality in \eqref{eqn:lfls}. The second equality follows analogously.
\end{proof}

\begin{figure}[h]
	\centering
	\includegraphics[scale=1]{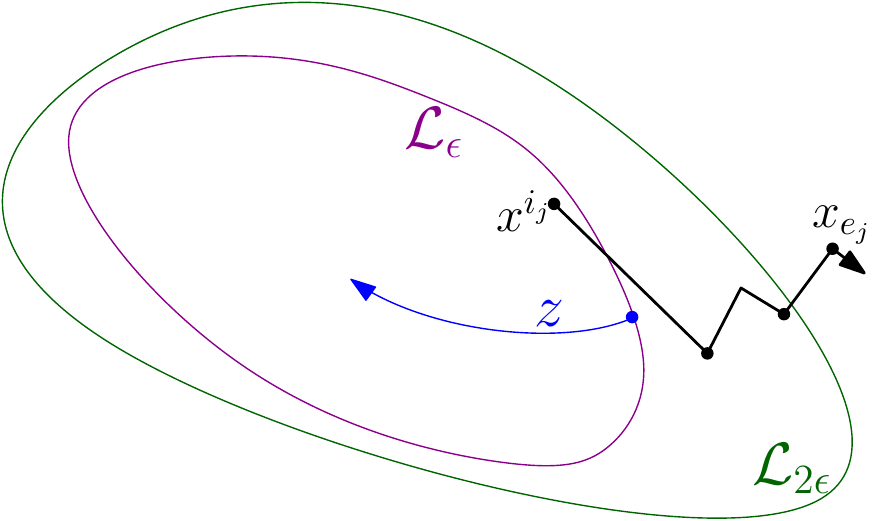}
	\caption{Illustration of the non-escape argument}
	\label{fig:attractor}
\end{figure}

The proof of Theorem~\ref{sec:proof_main_thm} will follow quickly from the following 
proposition. 
\begin{proposition}\label{prop:limit_of_functions}
The values $\varphi(x(t))$ have a limit as $t\to\infty$.
\end{proposition}
\begin{proof}
Without loss of generality, suppose $0=\lf_{t\to\infty} \varphi(x(t))$. 
For each $r\in \R$, define the sublevel set $$\mathcal{L}_{r}:=\{x\in\R^d: \varphi(x)\leq r\}.$$

Choose any $\epsilon >0$ satisfying $ \epsilon \notin  \varphi( G^{-1}(0))$. Note that by Assumption~\ref{ass:b_general}, we can let $\epsilon>0$ be as small as we wish.
By the first equality in~\eqref{eqn:lfls}, there are infinitely many indices $k$ such that $\varphi(x_k) < \epsilon$. The following elementary observation shows that for all large $k$, if $x_k$ lies in $\mathcal{L}_{\epsilon}$ then the next iterate $x_{k+1}$ lies in $\mathcal{L}_{2\epsilon}$.

\begin{claim}\label{claim:rand_claim_on2}
For all sufficiently large indices $k\in \mathbb{N}$, the implication holds:
$$x_k\in \mathcal{L}_{\epsilon}\quad \Longrightarrow\quad x_{k+1}\in \mathcal{L}_{2\epsilon}.$$
\end{claim}	
\begin{proof}
Since the sequence $\{x_k\}_{k\geq 1}$ is bounded, it is contained in some compact set $C\subset\R^n$. From continuity, we have
$$\cl(\R^d\setminus \mathcal{L}_{2\epsilon})= \cl(\varphi^{-1}(2\epsilon,\infty))\subseteq \varphi^{-1}[2\epsilon,\infty).$$ It follows that the two closed sets, $C\cap \mathcal{L}_{\epsilon}$ and $\cl(\R^d\setminus \mathcal{L}_{2\epsilon})$, do not intersect. Since $C\cap \mathcal{L}_{\epsilon}$ is compact, we deduce that it is well separated from $\R^d\setminus \mathcal{L}_{2\epsilon}$; that is, there exists $\alpha>0$ satisfying:
$$\min\{\|w-v\|: w\in C\cap \mathcal{L}_{\epsilon},~v\notin \mathcal{L}_{2\epsilon}\}\geq \alpha>0.$$
In particular $\dist(x_k;\R^d\setminus \mathcal{L}_{2\epsilon})\geq \alpha>0$, whenever $x_k$ lies in $\mathcal{L}_{\epsilon}$. Taking into account Lemma~\ref{lem:step_size}, we deduce $\|x_{k+1}-x_k\|<\alpha$ for all large $k$, and therefore   $x_k\in \mathcal{L}_{\epsilon}$ implies $x_{k+1}\in \mathcal{L}_{2\epsilon}$, as claimed.
\end{proof}

Let us define now the following sequence of iterates. Let $i_1\in \mathbb{N}$ be the first index satisfying
\begin{enumerate}
\item $x_{i_{1}}\in \mathcal{L}_{\epsilon}$, 
\item $x_{i_{1} + 1}\in \mathcal{L}_{2\epsilon}\setminus \mathcal{L}_{\epsilon}$, and 
\item defining the exit time $e_1:=\min\{e\geq i_1: x_e\notin \mathcal{L}_{2\epsilon}\setminus \mathcal{L}_{\epsilon}\}$, the iterate $x_{e_1}$ lies in $\R^d\setminus \mathcal{L}_{2\epsilon}$.
\end{enumerate} 
Then let $i_2>i_1$ be the next smallest index satisfying the same property, and so on. See Figure~\ref{fig:attractor} for an illustration. The following claim will be key.

\begin{claim}\label{claim:finite_ternminate}
This process must terminate, that is $\{x_k\}$ exits $\mathcal{L}_{2\epsilon}$ only finitely many times.
\end{claim}

Before proving the claim, let us see how it immediately yields the validity of the theorem. To this end, observe that 
 Claims~\ref{claim:rand_claim_on2} and \ref{claim:finite_ternminate} immediately imply $x_k\in \mathcal{L}_{2\epsilon}$ for all large $k$. Since $\epsilon>0$ can be made arbitrarily small, we deduce $\lim_{k\to\infty} \varphi(x_k)=0$. Equation~\eqref{eqn:lfls} then directly implies $\lim_{t\to\infty} \varphi(x(t))=0$, as claimed.

\begin{proof}[Proof of Claim~\ref{claim:finite_ternminate}]
To verify the claim, suppose that the process does not terminate. Thus we obtain an increasing sequence of indices $i_j\in \mathbb{N}$ with $i_j\to\infty$ as $j\to \infty$. Set $\tau_j=t_{i_j}$
and consider the curves $x^{\tau_j}(\cdot)$ in $\mathcal{C}(\R_+,\R^d)$. Then up to a subsequence, Theorem~\ref{thm:duchi_arzela} shows that the curves $x^{\tau_j}(\cdot)$ converge in $\mathcal{C}(\R_+,\R^d)$ to some arc $z(\cdot)$ satisfying $$\dot{z}(t)\in  G(z(t))\qquad \textrm{for a.e. } t\geq0.$$
By construction, we have $\varphi(x_{i_j})\leq \epsilon$ and $\varphi(x_{i_j+1})>\epsilon$. We therefore deduce
\begin{equation}\label{eqn:rand_ineq_cont}
\epsilon\geq \varphi(x_{i_{j}})\geq \varphi(x_{i_{j}+1})+(\varphi(x_{i_{j}})-\varphi(x_{i_{j}+1}))\geq \epsilon+[\varphi(x_{i_{j}})-\varphi(z(0))]-[\varphi(x_{i_{j}+1})-\varphi(z(0))].
\end{equation}
Recall $x_{i_{j}}\to z(0)$ as $j\to\infty$.
Lemma~\ref{lem:step_size} in turn implies $\|x_{i_{j}}-x_{i_{j}+1}\|\to 0$  and therefore $x_{i_{j}+1}\to z(0)$ as well. Continuity of $\varphi$ then guarantees that the right-hand-side of \eqref{eqn:rand_ineq_cont} tends to $\epsilon$, and hence $\varphi(z(0))=\lim_{j\to\infty} \varphi(x_{i_j})= \epsilon$. In particular, $z(0)$ is not an equilibrium point of $G$. Hence, Assumption~\ref{ass:b_general} yields a real $T>0$ such that 
$$\varphi(z(T)) < \sup_{t\in [0,T]} \varphi(z(t))  \leq \varphi(z(0)) = \epsilon.$$ 
In particular, there exists a real $\delta > 0$ satisfying 
$
\varphi(z(T)) \leq \epsilon - 2\delta.
$

Appealing to uniform convergence on $[0,T]$, we conclude
$$\sup_{t\in [0,T]}|\varphi(z(t))-\varphi(x^{\tau_j}(t))|<  \epsilon,$$
for all large $j\in\mathbb{N}$,
and therefore $$\sup_{t\in [0, T]} \varphi(x^{\tau_j}(t))\leq  \sup_{t\in [0, T]} \varphi(z(t))+\sup_{t\in [0, T]} |\varphi(z(t))-\varphi(x^{\tau_j}(t))|\leq 2\epsilon.$$ 
Hence, for all large $j$,
 all the curves $x^{\tau_j}$ map $[0,T]$ into $\mathcal{L}_{2\epsilon}$. We conclude that the exit time satisfies $$t_{e_j} > \tau_j + T \qquad \text{for all large $j$}.$$ We will show that the bound $\varphi(z(T)) \leq \epsilon - 2\delta$ yields the opposite inequality $t_{e_j} \leq \tau_j + T$, which will lead to a contradiction.

To that end, let 
$$
\ell_j = \max\{ \ell \in \NN \mid \tau_j \leq t_{\ell} \leq \tau_j + T\},
$$
be the last discrete index before $T$. Because $\alpha_k \rightarrow 0$ as $k \rightarrow \infty$, we have that $\ell_j \geq i_j +1$ for all large $j$.
We will now show that for all large $j$, we have $$\varphi(x_{{\ell_j}}) < \epsilon - \delta,$$ which implies $t_{e_j} < t_{\ell_j} \leq \tau_j +T$.
Indeed, observe 
$$\|x_{\ell_j}-x^{\tau_j}(T)\|= \|x^{\tau_j}(t_{\ell_j}-\tau_j)-x^{\tau_j}(T)\|\leq\|x_{\ell_j}-x_{\ell_j+1}\|\to 0.$$ 
Hence $x_{\ell_j}\to z(T)$ as $j\to\infty$. Continuity of $\varphi$ then guarantees  $\lim_{j\to\infty}\varphi(x_{\ell_j})=\varphi(z(T))$.
 Consequently, the inequality $\varphi(x_{\ell_j})<\epsilon - \delta$ holds for all large $j$, which is the desired contradiction.
\end{proof}
The proof of the lemma is now complete.
\end{proof}

We can now prove the main convergence theorem.

\begin{proof}[Proof of Theorem~\ref{thm:main_cray_thm_general}]
Let $x^*$ be a limit point of $\{x_k\}$ and suppose for the sake of contradiction that $0 \notin G(x^*)$. Let $i_j$ be the indices satisfying $x_{i_j}\to x^*$ as $j\to\infty$. Let $ z(\cdot)$ be the subsequential limit of the curves $x^{t_{i_j}}(\cdot)$ in $\mathcal{C}(\R_+,\R^d)$ guaranteed to exist by Theorem~\ref{thm:duchi_arzela}. Assumption~\ref{ass:b_general}  guarantees that there exists a real $T>0$ satisfying
$$\varphi(z(T))<\sup_{t\in [0,T]} \varphi(z(t))\leq \varphi(x^*).$$ 
 On the other hand, we successively deduce $$\varphi(z(T))=\lim_{j\to\infty} \varphi(x^{t_{i_j}}(T)) = \lim_{t \rightarrow \infty} \varphi(x(t)) =\varphi(x^*),$$ 
 where the last two equalities follow from Proposition~\ref{prop:limit_of_functions} and continuity of $\varphi$.
 We have thus arrived at a contradiction, and the theorem is proved.
\end{proof}

\section{Subgradient dynamical system }\label{section:subgradient}
Assumptions \ref{ass:duchi_assumption} and \ref{ass:b_general}, taken together, provide a powerful framework for proving subsequential convergence of algorithms to a zero of the set-valued map $G$. Note that the two assumptions are qualitatively different. Assumption~\ref{ass:duchi_assumption} is a property of both the algorithm \eqref{eqn:Euler_rec}  and the map $G$, while Assumption \ref{ass:b_general} is a property of $G$ alone.

For the rest of our discussion, we apply the differential inclusion approach outlined above to optimization problems.
Setting the notation, consider the optimization task
\begin{equation}\label{ean:main_opt_prob}
\min_{x\in\R^d}~f(x),
\end{equation}
where $f\colon\R^d\to\R$ is a locally Lipschitz continuous function. Seeking to apply the techniques of Section~\ref{differential_inclusions}, we simply set $G=-\partial f$ in the notation therein. Thus we will be interested in algorithms that, under reasonable conditions, track solutions of the differential inclusion
\begin{equation} \label{eqn:subgrad_follow}
\dot{z}(t)\in -\partial f(z(t))\qquad \textrm{for a.e. }t\geq 0,
\end{equation}
and subsequentially converge to critical points of $f$.
Discrete processes of the type \eqref{eqn:Euler_rec} for the optimization problem \eqref{ean:main_opt_prob} are often called stochastic approximation algorithms.  Here we study two such prototypical methods: the stochastic subgradient method in this section and the stochastic proximal subgradient in Section \ref{sec:prox_ext}. Each fits under the umbrella of Assumption \ref{ass:duchi_assumption}. 

Setting the stage, the {\em stochastic subgradient method} simply iterates the steps:
\begin{equation}\label{eqn:Euler_rec_clarke}
x_{k+1}=x_k-\alpha_k (y_k +  \xi_{k})\qquad \textrm{with}\qquad y_k\in \partial f(x_k),
\end{equation}
where $\{\alpha_k\}_{k\geq 1}$ is a step-size sequence and $\{\xi_k\}_{k\geq 1}$ is now a sequence of random variables (the ``noise'') on some probability space. 
Let us now isolate the following standard assumptions (e.g. \cite{Borkar,kush_yin}) for the method and see how they immediately imply Assumption \ref{ass:duchi_assumption}.

\begin{assumption}[Standing assumptions for the stochastic subgradient method]\label{ass:stand_ass}{\hfill }
\begin{enumerate}
\item\label{it:steps2_summable1} The sequence $\{\alpha_k\}$ is nonnegative, square summable, but not summable: $$\alpha_k\geq 0, \qquad\sum_{k=1}^{\infty}\alpha_k=\infty,\qquad \textrm{and}\qquad \sum_{k=1}^{\infty}\alpha_k^2<\infty.$$
\item \label{it:bound_iter1} Almost surely, the stochastic subgradient iterates are bounded: $\sup_{k\geq 1} \|x_k\|<\infty$.
\item\label{it:martingale_diff1} $\{\xi_k\}$ is a martingale difference sequence w.r.t the increasing $\sigma$-fields $$\mathcal{F}_k=\sigma(x_j,y_j,\xi_j: j\leq k).$$
That is, there exists a function $p: \RR^d \rightarrow [0, \infty)$, which is bounded on bounded sets, so that almost surely, for all $k\in \mathbb{N}$,  we have $$\EE[\xi_{k}|\mathcal{F}_k]=0 \qquad\textrm{ and }\qquad \EE[\|\xi_{k}\|^2|\mathcal{F}_k]\leq p(x_k).$$
\end{enumerate}
\end{assumption}

The following is true.

\begin{lem}
	Assumption~\ref{ass:stand_ass} guarantees that almost surely	Assumption~\ref{ass:duchi_assumption} holds.
\end{lem}
\begin{proof}
Suppose Assumption~\ref{ass:stand_ass} holds.
Clearly \ref{ass:duchi_assumption}.\ref{item:iterates_in_Q} and \ref{ass:duchi_assumption}.\ref{item:steps_duchi} hold vacuously, while \ref{ass:duchi_assumption}.\ref{item:gradients_bounded} follows immediately from \ref{ass:stand_ass}.\ref{it:bound_iter1} and local Lipschitz continuity of $f$. Assumption~\ref{ass:duchi_assumption}.\ref{item:average_OSC} follows quickly from the fact the $\partial f$ outer-semicontinuous and compact-convex valued; we leave the details to the reader. 
Thus we must only verify \ref{ass:duchi_assumption}.\ref{item:noise}, which follows quickly from standard martingale arguments. Indeed,
	notice from Assumption~\ref{ass:stand_ass}, we have 
	\begin{align*}
	\EE\left[\xi_{k} \mid \cF_k\right] = 0 \quad \forall k && \text{and} && \sum_{i=0}^\infty \alpha_i^2 \EE\left[\|\xi_{i}\|^2\mid  \cF_i \right] \leq  \sum_{i=0}^\infty \alpha_i^2p(x_i) < \infty. 
	\end{align*}	
	Define the $L^2$ martingale $X_k = \sum_{i=1}^k \alpha_i \xi_{i}$. Thus the limit $\dotp{X}_{\infty}$ of the predictable compensator 
	$$
	\dotp{X}_k := \sum_{i=1}^k \alpha_i^2 \EE\left[\|\xi_{i}\|^2\mid  \cF_i \right],
	$$
	exists. Applying \cite[Theorem 5.3.33(a)]{dembo2016probability}, we deduce that almost surely $X_k$ converges to a finite limit.
\end{proof}

Thus applying Theorem~\ref{thm:duchi_arzela}, we deduce that under Assumption~\ref{ass:stand_ass}, almost surely, the stochastic subgradient path  tracks a trajectory of the differential inclusion \eqref{eqn:subgrad_follow}. As we saw in Section~\ref{differential_inclusions}, proving subsequential convergence to critical points requires existence of a Lyapunov-type function $\varphi$ for the continuous dynamics. Henceforth, let us assume that the Lyapunov function $\varphi$ is $f$ itself. Section~\ref{sec:verifying} is devoted entirely to justifying this assumption for two broad classes of functions that are virtually exhaustive in data scientific contexts.

\begin{assumption}[Lyapunov condition in unconstrained minimization]\label{ass:b}\hfill
	\begin{enumerate}
		\item\label{it:dense} {\bf (Weak Sard)} The set of noncritical values of $f$ is dense in $\R$.
		\item\label{ass:b22} {\bf (Descent)}	Whenever $z\colon\R_+\to\R^d$ is trajectory of the differential inclusion $\dot{z}\in -\partial f(z)$ and $z(0)$ is not a critical point of $f$, there exists a real $T>0$ satisfying 
		$$f(z(T))<\sup_{t\in [0,T]} f(z(t))\leq f(z(0)).$$ 
	\end{enumerate}	
\end{assumption}

Some comments are in order. Recall that the classical Sard's theorem guarantees that the set of critical values of any $C^d$-smooth function $f\colon\R^d\to\R$ has measure zero. Thus property~\ref{it:dense} in Assumption~\ref{ass:b} asserts a very weak version of a nonsmooth Sard theorem. This is a very mild property, there mostly for technical reasons. It can fail, however, even for a $C^1$ smooth function on $\R^2$; see the famous example of Whitney \cite{whitney1935}. Property~\ref{ass:b22} of Assumption~\ref{ass:b} is more meaningful. It essentially asserts that $f$ must locally strictly decrease along any subgradient trajectory emanating from a noncritical point. 

Thus applying Theorem~\ref{thm:main_cray_thm_general}, we have arrived at the following guarantee for the stochastic subgradient method.

\begin{thm}\label{thm:main_cray_thm_subgradient_case}
	Suppose that Assumptions \ref{ass:stand_ass} and \ref{ass:b} hold. Then almost surely, every limit point of stochastic subgradient iterates $\{x_{k}\}_{k\geq 1}$ is critical for $f$ and the function values $\{f(x_k)\}_{k\geq 1}$ converge.
\end{thm}

\section{Verifying the descent condition}\label{sec:verifying}
In light of Theorems~\ref{thm:main_cray_thm_general} and \ref{thm:main_cray_thm_subgradient_case}, it is important to isolate a class of functions that automatically satisfy Assumption~\ref{ass:b}.\ref{ass:b22}. In this section, we do exactly that, focusing on two problem classes: (1) subdifferentially regular functions and (2) those functions whose graphs are Whitney stratifiable. We will see that the latter problem class also  satisfies \ref{ass:b}.\ref{it:dense}.

The material in this section is not new. In particular, the results of this section have appeared in \cite[Section 5.1]{DIL}. These results, however, are somewhat hidden in the paper \cite{DIL} and are difficult to parse. Moreover, at the time of writing \cite[Section 5.1]{DIL}, there was no clear application of the techniques, in contrast to our current paper. Since we do not expect the readers to be experts in variational analysis and semialgebraic geometry, we provide here a self-contained treatment, highlighting only the most essential ingredients and streamlining some of the arguments.

Let us begin with the following definition, whose importance for verifying Property~\ref{ass:b22} in Assumption~\ref{ass:b} will become clear shortly. 
\begin{definition}[Chain rule]\label{def:chain_rule}
{\rm
Consider a locally Lipschitz function $f$ on $\R^d$. We will say that {\em $f$ admits a chain rule} if for any arc $z\colon\R_+\to\R^d$,  equality 
$$(f\circ z)'(t)=\langle \partial f(z(t)),\dot{z}(t) \rangle \qquad \textrm{holds for a.e. }t \geq 0.$$
}
\end{definition}

The importance of the chain rule becomes immediately clear with the following lemma.

\begin{lem}\label{lem:chaion_rule_gives_assump}
Consider a locally Lipschitz function $f\colon\R^d\to\R$ that admits a chain rule. Let $z\colon\R_+\to\R^d$ be any arc satisfying the differential inclusion $$\dot{z}(t)\in -\partial f(z(t))\qquad \textrm{for a.e. }t\geq 0.$$ Then equality $\|\dot{z}(t)\|=\dist(0,\partial f(z(t)))$ holds for a.e. $t\geq 0$, and therefore 
	\begin{equation}\label{eqn:lyap}
	f(z(0))-f(z(t))= \int_{0}^t \dist^2\left(0;\partial f(z(\tau))\right)\,d\tau,\qquad \forall t\geq 0.
	\end{equation}
	In particular, property \ref{ass:b22} of Assumption~\ref{ass:b} holds.
\end{lem}
\begin{proof}
Fix a real $t\geq 0$ satisfying $(f\circ z)'(t)=\langle \partial f(z(t)),\dot{z}(t)\rangle$ . Observe then the equality
\begin{equation}\label{eqn:get_min_grad}
0=\langle \partial f(z(t))-\partial f(z(t)),\dot{z}(t)\rangle.
\end{equation}
To simplify the notation, set $S:=\partial f(z(t))$, $W:=\spann(S-S)$, and $y:=-\dot{z}(t)$. Appealing to \eqref{eqn:get_min_grad}, we conclude $y\in W^{\perp}$, and therefore trivially we have
$$y\in (y+W)\cap W^{\perp}.$$
Basic linear algebra implies $\|y\|=\dist(0;y+W)$. Noting $\partial f(z(t))\subset y+W$, we deduce $\|\dot{z}(t)\|\leq \dist(0;\partial f(z(t)))$ as claimed. Since the reverse inequality trivially holds, we obtain the claimed equality, $\|\dot{z}(t)\|= \dist(0;\partial f(z(t)))$.

Since $f$ admits a chain rule, we conclude for a.e. $\tau\geq 0$ the estimate
$$(f\circ z)'(\tau)=\langle \partial f(z(\tau)),\dot{z}(\tau) \rangle=-\|\dot z(\tau)\|^2 =- \dist^2\left(0;\partial f(z(\tau))\right).$$
Since $f$ is locally Lipschitz, the composition $f\circ z$ is absolutely continuous. Hence integrating over the interval $[0,t]$  yields \eqref{eqn:lyap}.

Suppose now that the point $z(0)$ is noncritical. Then by outer semi-continuity of $\partial f$, the exists $T>0$ such that $z(\tau)$ is noncritical for any $\tau\in [0,T]$. It follows immediately that the value $ \int_{0}^t \dist^2\left(0;\partial f(z(\tau))\right)\,d\tau$ is strictly increasing in $t\in [0,T]$, and therefore by \eqref{eqn:lyap} that $f\circ z$ is strictly decreasing. Hence item \ref{ass:b22} of Assumption~\ref{ass:b} holds, as claimed.
\end{proof}

Thus property \ref{ass:b22} of Assumption~\ref{ass:b} is sure to hold as long as $f$ admits a chain rule. 
In the following two sections, we identify two different function classes that indeed admit the chain rule.

\subsection{Subdifferentially regular functions}\label{subsec:subdiff_reg}
The first function class we consider consists of subdifferentially regular functions.
Such functions play a prominent role in variational analysis due to their close connection with convex functions; we refer the reader to the monograph \cite{RW98} for details. In essence, subdifferential regularity  forbids downward facing cusps in the graph of the function; e.g. $f(x)=-|x|$ is not subdifferentially regular. We now present the formal definition. 
\begin{definition}[Subdifferential regularity]
{\rm
A locally Lipschitz function $f\colon\R^d\to\R$ is {\em subdifferentially regular} at a point $x\in\R^d$ if every subgradient $v\in\partial f(x)$ yields an affine minorant of $f$ up to first-order:
$$f(y)\geq f(x)+\langle v,y-x\rangle+o(\|y-x\|) \qquad \textrm{as } y\to x.$$ }
\end{definition}

The following lemma shows that any locally Lipschitz function that is subdifferentially regular  indeed admits a chain rule.

\begin{lem}[Chain rule under subdifferential regularity]\label{lem:chain_rule_subdiff_reg}
Any locally Lipschitz function that is subdifferentially regular  admits a chain rule and therefore item \ref{ass:b22} of Assumption~\ref{ass:b} holds.
\end{lem}
\begin{proof}
	Let $f\colon\R^d\to\R$ be a locally Lipschitz and subdifferentially regular function.
Consider an arc $x\colon\R_+\to\R^d$. Since, $x$ and $f\circ x$ are absolutely continuous, both are differentiable almost everywhere. Then for any such $t\geq 0$ and any subgradient $v\in \partial f(x(t))$, we conclude
\begin{align*}
(f\circ x)'(t) =\lim_{r\searrow 0} \frac{f(x(t+r))-f(x(t))}{r} &\geq \lim_{r\searrow 0}\frac{\langle v,x(t+r)-x(t)\rangle+o(\|x(t+r)-x(t)\|)}{r}\\
&= \langle v,\dot x(t)\rangle.
\end{align*}
Instead, equating $(f\circ x)'(t)$ with the left limit of the difference quotient yields the reverse inequality $(f\circ x)'(t)\leq \langle v,\dot x(t)\rangle$. Thus $f$ admits a chain rule and  item \ref{ass:b22} of Assumption~\ref{ass:b} holds by  Lemma~\ref{lem:chaion_rule_gives_assump}.
\end{proof}

Thus we have arrived at the following corollary. For ease of reference, we state subsequential convergence guarantees both for the general process \eqref{eqn:Euler_rec}  and for the specific stochastic subgradient method \eqref{eqn:Euler_rec_clarke}.

\begin{cor}
	Let $f\colon\R^d\to\R$ be a locally Lipschitz function that is subdifferentially regular and such that its set of noncritical values is dense in $\R$. 
	\begin{itemize}
		\item {\bf (Stochastic approximation)} Consider the iterates $\{x_k\}_{k \geq 1}$ produced by \eqref{eqn:Euler_rec} and suppose  that Assumption~\ref{ass:duchi_assumption} holds with $G=-\partial f$. Then every limit point of  the iterates $\{x_{k}\}_{k\geq 1}$ is critical for $f$ and the function values $\{f(x_k)\}_{k\geq 1}$ converge.
		\item {\bf (Stochastic subgradient method)}
		Consider the iterates $\{x_k\}_{k \geq 1}$ produced by the stochastic subgradient method \eqref{eqn:Euler_rec_clarke} and suppose that
		 Assumption~\ref{ass:stand_ass} holds. Then almost surely, every limit point of  the iterates $\{x_{k}\}_{k\geq 1}$ is critical for $f$ and the function values $\{f(x_k)\}_{k\geq 1}$ converge. 
	\end{itemize}
\end{cor}

Though subdifferentially regular functions are widespread in applications, they preclude ``downwards cusps'', and therefore do not capture such simple examples as $f(x,y) = (|x|-|y|)^2$ and $f(x) = (1-\max\{x,0\})^2$. The following section concerns a different function class that does capture these two nonpathological examples.

\subsection{Stratifiable functions}
As we saw in the previous section, subdifferential regularity is a local property  that implies the desired item \ref{ass:b22} of Assumption~\ref{ass:b}. In this section, we instead focus on a broad class of functions satisfying a global geometric property, which eliminates pathological examples from consideration.

Before giving a formal definition, let us fix some notation. A set $M\subset\R^d$ is a {\em $C^p$ smooth manifold} if there is an integer $r\in \mathbb{N}$ such that around any point $ x\in M$, there is a  neighborhood $U$ and a $C^p$-smooth map $F\colon U\to\R^{d-r}$ with $\nabla F( x)$ of full rank and satisfying $M\cap U=\{y\in U: F(y)=0\}$. If this is the case, the {\em tangent} and {\em normal spaces} to $M$ at $x$ are defined to be $T_M(x):=\Null(\nabla F(x))$ and $N_M(x):=(T_M(x))^{\perp}$, respectively.

\begin{definition}[Whitney stratification]
{\rm
A  {\em Whitney $C^p$-stratification} $\mathcal{A}$ of a set $Q\subset\R^d$ is a partition of $Q$ into finitely many nonempty $C^p$ manifolds, called {\em strata}, satisfying the following compatibility conditions.
\begin{enumerate}
\item {\bf Frontier condition:} For any two strata $L$ and $M$, the implication 
$$L\cap \cl M\neq \emptyset\qquad \Longrightarrow \qquad L\subset \cl M\qquad \textrm{holds}.$$
\item {\bf Whitney condition (a):} For any sequence of points $z_k$ in a stratum $M$ converging to a point $\bar z$ in a stratum $L$, if the corresponding normal vectors $v_k\in N_{M}(z_k)$ converge to a vector $v$, then the inclusion $v\in N_L(\bar z)$ holds. 
 \end{enumerate}
 A function $f\colon\R^d\to\R$ is {\em Whitney $C^p$-stratifiable} if its graph admits a Whitney $C^p$-stratification.
}
\end{definition}

The definition of the Whitney stratification invokes 
two conditions, one topological and the other geometric.
The frontier condition simply says that if one stratum $L$ intersects the closure of another $M$, then $L$ must be fully contained in the closure $\cl M$. In particular, the frontier condition endows the strata with a partial order $L\preceq M\Leftrightarrow L\subset \cl M$. The Whitney condition (a) is geometric. In short, it asserts that limits of normals along a sequence $x_i$ in a stratum $M$ are themselves normal to the stratum containing the limit of $x_i$.

The following discussion of Whitney stratifications follows that in \cite{bolte2007clarke}.
Consider a Whitney $C^p$-stratification $\{M_i\}$ of the graph of a locally Lipschitz function $f\colon\R^d\to\R$. Let $\{\mathcal{M}_i\}$ be the manifolds obtained by projecting $\{M_i\}$ on $\R^d$. An easy argument using the constant rank theorem shows that the partition $\{\mathcal{M}_i\}$ of $\R^d$ is itself a Whitney $C^p$-stratification and the restriction of $f$ to each stratum $\{\mathcal{M}_i\}$ is $C^p$-smooth. Whitney condition (a) directly yields the following consequence~\cite[Proposition 4]{bolte2007clarke}. For any stratum $\mathcal{M}$ and any point $x\in \mathcal{M}$, we have
\begin{equation}\label{eqn:proj_norm_form}
(v,-1)\in N_M(x,f(x))\qquad \textrm{for all}\quad v\in \partial f(x),
\end{equation}
and 
\begin{equation}\label{eqn:subdiff_inclusion}
\partial f(x)\subset \nabla g(x)+N_{\mathcal{M}}(x),
\end{equation}
 where $g\colon\R^d\to\R$ is any $C^1$-smooth function agreeing with $f$ on a neighborhood of $x$ in $\mathcal{M}$.

The following theorem, which first appeared in \cite[Corollary 5]{bolte2007clarke}, shows that Whitney stratifiable functions automatically satisfy the weak Sard property of Assumption~\ref{ass:b}. We present a quick argument  here for completeness.
It is worthwhile to mention that such a Sard type result holds more generally for any stratifiable set-valued map; see the original work \cite{sard_ioffe} or the monograph \cite[Section 8.4]{ioffe_book}.

\begin{lem}[Stratified Sard]\label{lem:strat_sard}
	The set of critical values of any Whitney $C^d$-stratifiable locally Lipschitz function $f\colon\R^d\to\R$ has zero measure. In particular, item~\ref{it:dense} of Assumption~\ref{ass:b} holds.
\end{lem}
\begin{proof}
Let $\{M_i\}$ be the strata of a Whitney $C^d$ stratification of the graph of $f$. Let $\pi_i\colon M_i\to \R$ be the restriction of the orthogonal projection $(x,r)\mapsto r$ to the manifold $M_i$. We claim that each critical value of $f$ is a critical value (in the classical analytic sense) of $\pi_i$, for some index $i$. 
To see this, consider a critical point $x$ of $f$ and let $M_i$ be the stratum of $\gph f$ containing $(x,f(x))$. Since $x$ is critical for $f$, appealing to \eqref{eqn:proj_norm_form} yields the inclusion $(0,-1)\in N_{M_i}(x,f(x))$ and therefore the equality $\pi_i\left(T_{M_i}(x,f(x))\right)=\{0\}\subsetneq \R$. Hence $(x,f (x))$ is a critical point of $\pi_i$ and $f(x)$ its critical value, thereby establishing the claim. Since the set of critical values of each map $\pi_i$ has zero measure by the standard Sard's theorem, and there are finitely many strata, it follows that the set of critical values of $f$ also has zero measure.
\end{proof}

Next, we prove the chain rule for  any Whitney stratifiable function.

\begin{thm}\label{thm:chain_strat}
Any locally Lipschitz function $f\colon\R^d\to\R$ that is Whitney $C^1$-stratifiable admits a chain rule, and therefore item \ref{ass:b22} of Assumption~\ref{ass:b} holds.
\end{thm}
\begin{proof}
Let $\{M_i\}$ be the Whitney $C^1$-stratification of $\gph f$ and let $\{\mathcal{M}_i\}$ be its coordinate projection onto $\R^d$. Fix an arc $x\colon\R^d\to\R$. Clearly, both $x$ and $f\circ x$ are differentiable at a.e. $t\geq 0$. Moreover, we claim that for a.e. $t\geq 0$, the implication holds:
\begin{equation}\label{eqn:implication_tan}
x(t)\in \mathcal{M}_i\quad \Longrightarrow \quad \dot{x}(t)\in T_{\mathcal{M}_i}(x(t)).
\end{equation}
To see this, fix a manifold $\mathcal{M}_i$ and let $\Omega_i$ be the set of all $t\geq 0$ such that $x(t)\in \mathcal{M}_i$, the derivative $\dot x(t)$ exists, and we have  $\dot{x}(t)\notin T_{\mathcal{M}_i}(x(t))$. If we argue that $\Omega_i$ has zero measure, then so does the union $\cup_i \Omega_i$ and the claim is proved. Fix an arbitrary $t\in \Omega_i$. There exists a closed interval $I$ around $t$ such that $x$ restricted to $I$ intersects $\mathcal{M}_i$ only at $x(t)$, since otherwise we would deduce that $\dot{x}(t)$ lies in $T_{\mathcal{M}_i}(x(t))$ by definition of the tangent space. We may further shrink $I$ such that its endpoints are rational. It follows that $\Omega_i$ may be covered by disjoint closed intervals with rational endpoints. Hence $\Omega_i$ is countable and therefore zero measure, as claimed.

Fix now a real $t>0$ such that $x$ and $f\circ x$ are differentiable at $t$ and the implication \eqref{eqn:implication_tan} holds. Let $\mathcal{M}$ be a stratum containing $x(t)$. Since $\dot{x}(t)$ is  tangent to $\mathcal{M}$ at $x(t)$, there exists a $C^1$-smooth curve $\gamma\colon (-1,1)\to\mathcal{M}$  satisfying $\gamma(0) = x(t)$ and $\dot\gamma(0)=\dot{x}(t)$. Let $g\colon\R^d\to\R$ be any $C^1$ function agreeing with $f$ on a neighborhood of $x(t)$ in $\mathcal{M}$. We claim that $(f\circ x)'(t)=(f\circ \gamma)'(0)$. Indeed, letting $L$ be a Lipschitz constant of $ f$ around $x(t)$, we deduce
\begin{align*}
(f\circ x)'(t) = \lim_{ r \rightarrow 0} \frac{ f(x(t+r)) - f(x(t))}{r} 
&= \lim_{ r \rightarrow 0} \frac{ f(x(t+r)) - f(\gamma(r)) + f(\gamma(r))  - f(\gamma(0))}{r} \\
&= \lim_{ r \rightarrow 0} \frac{ f(x(t+r)) - f(\gamma(r))}{r} + (f\circ \gamma)'(0).
\end{align*}
Notice  $ \tfrac{ |f(x(t+r)) - f(\gamma(r))|}{r} \leq  L\left\|\tfrac{x(t+r) - x(t)  + \gamma(0) -  \gamma(r)}{r}\right\| \rightarrow L\|\dot x(t) - \dot \gamma(0)\| = 0 $ as $ r\rightarrow 0.$
Thus 
$$(f\circ x)'(t)=(f\circ \gamma)'(0)=(g\circ \gamma)'(0)=\langle \nabla g(x),\dot\gamma(0)\rangle=\langle\partial f(x(t)),\dot{x}(t)\rangle,$$
where the last equality follows from \eqref{eqn:subdiff_inclusion}.
\end{proof}

Putting together Theorems~\ref{thm:main_cray_thm_general}, \ref{thm:main_cray_thm_subgradient_case}, \ref{thm:chain_strat} and Lemma~\ref{lem:strat_sard}, we arrive at the main result of our paper. Again for ease of reference, we state subsequential convergence guarantees both for the general process \eqref{eqn:Euler_rec}  and for the specific stochastic subgradient method \eqref{eqn:Euler_rec_clarke}.
\begin{cor}
Let $f\colon\R^d\to\R$ be a locally Lipschitz function that is $C^d$-stratifiable.
\begin{itemize}
	\item {\bf (Stochastic approximation)} Consider the iterates $\{x_k\}_{k \geq 1}$ produced by \eqref{eqn:Euler_rec} and suppose  that Assumption~\ref{ass:duchi_assumption} holds with $G=-\partial f$. Then every limit point of  the iterates $\{x_{k}\}_{k\geq 1}$ is critical for $f$ and the function values $\{f(x_k)\}_{k\geq 1}$ converge.
	\item {\bf (Stochastic subgradient method)}
	Consider the iterates $\{x_k\}_{k \geq 1}$ produced by the stochastic subgradient method \eqref{eqn:Euler_rec_clarke} and suppose that
	Assumption~\ref{ass:stand_ass} holds. Then almost surely, every limit point of  the iterates $\{x_{k}\}_{k\geq 1}$ is critical for $f$ and the function values $\{f(x_k)\}_{k\geq 1}$ converge. 
\end{itemize}
\end{cor}

Verifying Whitney stratifiability is often an easy task. Indeed, there are a number of well-known and easy to recognize function classes, whose members are automatically Whitney stratifiable. We now briefly review such classes, beginning with the semianalytic setting.

A closed set $Q$ is called {\em semianalytic} if it can be written as a finite union of sets, each having the form
$$\{x\in\R^d:p_{i}(x)\leq 0\quad \textrm{for }i=1,\ldots, \ell\}$$
for some real-analytic functions $p_1, p_2,\ldots,p_{\ell}$ on $\R^d$. 
If the functions $p_1, p_2,\ldots,p_{\ell}$ in the description above are polynomials, then $Q$ is said to be a {\em semialgebraic} set. A well-known result of {\L}ojasiewicz \cite{loja_semi} shows that any semianalytic set admits a Whitney $C^{\infty}$ stratification. 
Thus,  the results of this paper apply to functions with semianalytic graphs. While the class of such functions is broad, it is sometimes difficult to recognize its members  
as semianalyticity is not preserved under some basic operations, such as projection onto a linear subspace. On the other hand, there are large subclasses of semianalytic sets that are easy to recognize.

For example, every semialgebraic set is semianalytic, but in contrast to the semianalytic case, semi-algebraic sets are stable with respect to all boolean operations and projections onto subspaces. The latter property is a direct consequence of the celebrated Tarski-Seidenberg Theorem. Moreover, semialgebraic sets are typically easy to recognize using quantifier elimination; see \cite[Chapter 2]{Coste-semi} for a detailed discussion. Importantly, compositions of semialgebraic functions are semialgebraic. 

A far reaching axiomatic extension of semialgebraic sets, whose members are also Whitney stratifiable, is built from ``o-minimal structures''. Loosely speaking, sets that are definable in an o-minimal structure share the same robustness properties and attractive analytic features as semialgebraic sets. For the sake of completeness, let us give a formal definition, following Coste \cite{Coste-min} and van den Dries-Miller\cite{Dries-Miller96}.

\begin{definition}[o-minimal structure]
{\rm
An {\em o-minimal structure} is a sequence of Boolean algebras $\mathcal{O}_d$ of subsets of $\R^d$ such that for each $d\in \mathbb{N}$:
\begin{enumerate}[(i)]
\item if $A$ belongs to $\mathcal{O}_d$, then $A\times \R$ and $\R\times A$ belong to $\mathcal{O}_{d+1}$;
\item if $\pi\colon\R^{d}\times \R\to\R^d$ denotes the coordinate projection onto $\R^d$, then for any $A$ in $\mathcal{O}_{d+1}$ the set $\pi(A)$ belongs to  $\mathcal{O}_{d}$;
\item  $\mathcal{O}_{d}$ contains all sets of the form $\{x\in \R^d: p(x)=0\}$, where $p$ is  a polynomial on $\R^d$;
\item the elements of $\mathcal{O}_{1}$ are exactly the finite unions of intervals (possibly infinite) and points.
\end{enumerate}
The sets $A$ belonging to $\mathcal{O}_{d}$, for some $d\in \mathbb{N}$, are called {\em definable in the o-minimal structure}.}
\end{definition}

As in the semialgebraic setting, any function definable in an o-minimal structure admits a Whitney $C^p$ stratification, for any $p\geq 1$ (see e.g. \cite{Dries-Miller96}). Beyond semialgebraicity, Wilkie showed that that there is an o-minimal structure that simultaneously contains both the graph of the exponential function $x\mapsto e^x$ and all semi-algebraic sets~\cite{exp_omin}. 

\subsubsection*{A corollary for deep learning}
Since the composition of two definable functions is definable, we conclude that nonsmooth deep neural networks built from definable pieces---such as ReLU, quadratics $t^2$, hinge losses $\max\{0,t\}$, and SoftPlus $\log(1+e^t)$ functions---are themselves definable. Hence, the results of this paper endow stochastic subgradient methods, applied to definable deep networks, with rigorous convergence guarantees. Due to the importance of subgradient methods in deep learning, we make this observation precise in the following corollary which provides a rigorous convergence guarantee for a wide class of deep learning loss functions that are recursively defined, including convolutional neural networks, recurrent neural networks, and feed-forward networks.

\begin{cor}[Deep networks]\label{cor:deep}
For each given data pair $(x_j,y_j)$ with $j=1,\ldots, n$, recursively define:
$$a_0=x_j, \qquad a_i=\rho_i(V_i(w)a_{i-1}) ~~\forall i=1,\ldots,L,\qquad f(w;x_j,y_j)=\ell(y_j,a_L),$$
where
\begin{enumerate}
	\item $V_i(\cdot)$ are linear maps into the space of matrices.
	\item $\ell(\cdot;\cdot)$ is any definable loss function, such as the logistic loss $\ell(y;z)=\log(1+e^{-yz})$, the hinge loss $\ell(y;z)=\max\{0,1-yz\}$, absolute deviation loss $\ell(y;z) = |y-z|$, or the square loss $\ell(y;z) = \frac12 (y-z)^2$.
	\item $\rho_i$ are definable activation functions applied coordinate wise, such as those whose domain can be decomposed into finitely many intervals on which it coincides with $\log t$, $\exp(t)$, $\max(0,t)$, or $\log(1+e^t)$.
\end{enumerate}
Let $\{w_k\}_{k\geq 1}$ be the iterates produced by the stochastic subgradient method on the deep neural network loss 
$
f(w):=\sum_{j=1}^n f(w;x_j,y_j)
$, and suppose that the standing assumption~\ref{ass:stand_ass} holds.\footnote{In the assumption, replace $x_k$ with $w_k$, since we now use $w_k$ to denote the stochastic subgradient iterates.} Then almost surely, every limit point $w^{*}$ of the iterates $\{w_k\}_{k\geq 1}$ is critical for $f$, meaning $0\in \partial f(w^*)$, and the function values $\{f(w_k)\}_{k\geq 1}$ converge.
\end{cor}

\section{Proximal extensions}\label{sec:prox_ext}

In this section, we extend most of the results in Sections~\ref{section:subgradient} and \ref{sec:verifying} on unconstrained problems to a ``proximal'' setting and comment on sufficient conditions to ensure boundedness of the iterates. The arguments follow quickly by combining the techniques developed by Duchi-Ruan \cite{duchi_ruan} with those presented in Section~\ref{sec:verifying}. Consequently, all the proofs are in Appendix~\ref{app}.

Setting the stage, consider the composite  optimization problem 
\begin{equation}\label{eqn:comp_prob}
\min_{x\in\cX}~ \varphi(x):=f(x) + g(x),
\end{equation}
where $f\colon\R^d\to\R$ and $g\colon\R^d\to\R$ are locally Lipschitz functions and $\cX$ is a closed set. 
As is standard in the literature on proximal methods, we will say that $x\in \cX$ is a {\em composite critical point} of the  problem \eqref{eqn:comp_prob} if the inclusion holds:
\begin{equation}\label{eqn:critical}
0\in \partial f(x)+\partial g(x)+N_{\cX}(x).
\end{equation}
Here, the symbol $N_{\cX}(x)$ denotes the Clarke normal cone to the closed set $\cX \subseteq \RR^d$ at $x\in \cX$. We refer the reader to Appendix~\ref{app} for a formal definition.
We only note that when $\cX$ is a closed convex set, $N_{\cX}$ reduces to the normal cone in the sense of convex analysis, while for a $C^1$-smooth manifold $\cX$, it reduces to the normal space in the sense of differential geometry. It follows from \cite[Corollary 10.9]{RW98} that local minimizers of \eqref{eqn:comp_prob} are necessarily composite critical.
A real number $r\in \R$ is called a {\em composite critical value} if equality, $r=f(x)+g(x)$, holds for some composite critical point $x$. 

Thus we will be interested in an extension of the stochastic subgradient method that tracks a trajectory of the differential inclusion
\begin{equation} 
\dot{z}(t)\in -(\partial f+\partial g+N_{\cX})(z(t))\qquad \textrm{for a.e. }t\geq 0,
\end{equation}
and subsequentially converges to a composite critical point of \eqref{eqn:comp_prob}. 
Seeking to apply the techniques of Section~\ref{differential_inclusions}, we can simply set $G=-\partial f-\partial g-N_{\cX}$ in the notation therein. Note that $G$ thus defined is not necessarily a subdifferential of a single function because equality in the subdifferential sum rule \cite[Corollary 10.9]{RW98}  can fail when the summands are not subdifferentially regular.

We now aim to describe the proximal stochastic subgradient method for the problem \eqref{eqn:comp_prob}. There are two ingredients we must introduce: a stochastic subgradient oracle for $f$ and the proximity map of $g+\delta_{\cX}$, where $\delta_{\cX}$ is the indicator function of $\cX$. We describe the two ingredients in turn.

\paragraph{Stochastic subgradient oracle}
Our model of the stochastic subgradient oracle follows that of the influential work \cite{robust_subgrad}.
 Fix a probability space $(\Omega,\mathcal{F},P)$ and equip $\R^d$ with the Borel $\sigma$-algebra. We suppose that there exists a measurable mapping $\zeta \colon \RR^d \times \Omega \rightarrow \RR^d$ satisfying: 
$$
\EE_{\omega}\left[\zeta(x,\omega)\right] \in \partial f(x)\qquad \textrm{for all }x\in \R^d.
$$
Thus  after sampling $\omega\sim P$, the vector $\zeta(x,\omega)$ can serve as a stochastic estimator for a true subgradient of $f$. 

\paragraph{Proximal map}
 Standard deterministic proximal splitting methods utilize the proximal map of $g+\delta_{\cX}$, namely:
$$z\mapsto \argmin_{x\in \cX}\{g(x)+\tfrac{1}{2\alpha}\|x-z\|^2\}.$$
Since we do not impose convexity assumptions on $g$ and $\cX$, this map can be set-valued. Thus we must pass to a measurable selection. Indeed, supposing that $g$ is bounded from below on $\cX$,  the result~\cite[Exercise 14.38]{RW98} guarantees that there exists a measurable selection $T_{(\cdot)}(\cdot) : (0, \infty) \times \RR^d \rightarrow \RR^d$, such that 
$$T_{\alpha}(z) \in \argmin_{x\in\mathcal{X}}\,\left\{g(x)+\tfrac{1}{2\alpha}\|x-z\|^2\right\} \qquad \text{for all }\alpha> 0, z \in \RR^d.$$ 

We can now formally state the algorithm. Given an iterate $x_k\in \cX$, the {\em proximal stochastic subgradient method} performs the update
\begin{equation}\label{eqn:stoc_prox_subgrad}
\left\{\begin{aligned}
&\textrm{Sample }\omega_k\sim P\\
&x_{k+1} = T_{\alpha_k}(x_k - \alpha_k\zeta(x_k, \omega_k)).
\end{aligned}\right\}.
\end{equation}
Here $\{\alpha_k\}_{k\geq 1}$ is a positive control sequence.  We will analyze the algorithm under the following two assumptions, akin to Assumptions~\ref{ass:stand_ass} and \ref{ass:b} of Section~\ref{section:subgradient}. 
Henceforth, define the set-valued map $G\colon\cX\rightrightarrows\R^d$ by
\begin{equation}\label{eqn:weird_composg}
 G(x)=-\partial f(x)-\partial g(x)-N_{\cX}(x).
\end{equation}

\begin{assumption}[Standing assumptions for the proximal stochastic subgradient method]\label{ass:stand_ass2}{\hfill }
\begin{enumerate}
\item\label{it:mild1} $\cX$ is closed, $f$ and $g$ are locally Lipschitz, and $g$ is bounded from below on $\cX$.
\item\label{it:mild2}  There exists a function $L\colon\R^d\to\R$, which is bounded on bounded sets, satisfying
$$L(x)\geq \sup_{z: g(z)\leq g(x)}\frac{g(x)-g(z)}{\|x-z\|}.$$
\item\label{it:steps2_summable} The sequence $\{\alpha_k\}_{k\geq 1}$ is nonnegative, square summable, but not summable: $$\alpha_k\geq 0, \qquad\sum_{k=1}^{\infty}\alpha_k=\infty,\qquad \textrm{and}\qquad \sum_{k=1}^{\infty}\alpha_k^2<\infty.$$
\item\label{it:mild4} Almost surely, the iterates are bounded: $\sup_{k \geq 1} \|x_k\| < \infty$.
\item\label{it:martingale_diff} 
There exists a function $p \colon \RR^d \rightarrow \RR_+$, that is bounded on bounded sets,  such that $$\EE_{\omega}\left[\zeta(x, \omega) \right] \in \partial f(x) \qquad\textrm{ and }\qquad \EE_{\omega}\left[\left\|\zeta(x, \omega)\right\|^2\right]\leq p(x)\qquad \textrm{for all }x\in \cX.$$
\item\label{it:loc_integrability} For every convergent sequence $\{z_k\}_{k \geq 1} $, we have
$$
\EE_{\omega}\left[\sup_{k \geq 1} \|\zeta(z_k,\omega)\|\right] < \infty.
$$

\end{enumerate}
\end{assumption}

\begin{assumption}[Lyapunov condition in proximal minimization]\label{ass:bhat}\hfill
	\begin{enumerate}
		\item\label{it:densehat} {\bf (Weak Sard)} The set of composite noncritical values of \eqref{eqn:comp_prob} is dense in $\R$.
		\item\label{ass:b2hat} {\bf (Descent)}	Whenever $z\colon\R_+\to\cX$ is an arc satisfying
		 the differential inclusion $$\dot{z}(t)\in -(\partial f+\partial g+N_{\cX})(z(t))\qquad \textrm{for a.e. }t\geq 0,$$ and $z(0)$ is not a composite critical point of \eqref{eqn:comp_prob}, there exists a real $T>0$ satisfying 
		$$\varphi(z(T))<\sup_{t\in [0,T]} \varphi(z(t))\leq \varphi(z(0)).$$ 
	\end{enumerate}	
\end{assumption}

Let us make a few comments. Properties \ref{ass:stand_ass2}.\ref{it:mild1}, \ref{ass:stand_ass2}.\ref{it:steps2_summable}, \ref{ass:stand_ass2}.\ref{it:mild4}, and \ref{ass:stand_ass2}.\ref{it:martingale_diff} are mild and completely expected in light of the results in the previous sections. Property~\ref{ass:stand_ass2}.\ref{it:loc_integrability} is a technical condition ensuring that the expected maximal noise in the stochastic subgradient is bounded along any convergent sequence. Finally, property \ref{ass:stand_ass2}.\ref{it:mild2} is a mild technical condition on function $g$ that we allow. In particular, it holds for any convex, globally Lipschitz, or coercive locally Lipschitz function. We record this observation in the following lemma.
\begin{lem}
Consider any function $g\colon\R^d\to\R_+$ that is either convex, globally Lipschitz, or locally Lipschitz and coercive. Then $g$ satisfies property \ref{it:mild2} in Assumption~\ref{ass:stand_ass2}.
\end{lem}
\begin{proof}
	Fix a point $x$ and a point $z$ satisfying $g(z)\leq g(x)$. 
Let us look at each case and upper bound the error $g(x)-g(z)$. Suppose first that $g$ is convex. Then for the vector $v\in \partial g(x)$ of minimal norm, we have
$$
g(x) - g(z) \leq  \langle v,x-z\rangle\leq \|v\|\|x-z\|= L(x)\cdot \|x-z\|.
$$
where $L(x):=\dist(0,\partial g(x))$.
If $f$ is globally Lipschitz, then clearly we have
$$
g(x) - g(z) \leq  L(x)\cdot\|x - z\|.
$$
where $L(x)$ is identically equal to the global Lipschitz constant of $g$.
Finally, in the third case, suppose that $g$ is coercive and locally Lipschitz. We deduce
$$
g(x) - g(z) \leq  L(x)\cdot\|x - z\|,
$$
where $L(x)$ is the Lipschitz modulus of $g$ on the compact sublevel set $[g\leq g(x)]$. Since $g$ is locally Lipschitz continuous, in all three cases, the function $L(\cdot)$ is bounded on bounded sets.	
\end{proof}


Under the two assumptions, \ref{ass:stand_ass2} and \ref{ass:bhat}, we obtain the following subsequential convergence guarantee. The argument in the appendix is an application of Theorem~\ref{thm:main_cray_thm_general}. To this end, we show that Assumption~\ref{ass:stand_ass2} implies Assumption~\ref{ass:duchi_assumption} almost surely, while Assumption~\ref{ass:bhat} is clearly equivalent to Assumption~\ref{ass:b_general}.
\begin{thm}\label{thm:main_cray_thmhat}
	Suppose that Assumptions \ref{ass:stand_ass2} and \ref{ass:bhat} hold. Then almost surely, every limit point of the iterates $\{x_{k}\}_{k\geq 1}$ produced by the proximal stochastic subgradient method \eqref{eqn:stoc_prox_subgrad}  is composite critical for \eqref{eqn:comp_prob} and the function values $\{\varphi(x_k)\}_{k\geq 1}$ converge.
\end{thm}

Finally, we must now understand problem classes that satisfy Assumption~\ref{ass:bhat}. To this end, analogously to Definition~\ref{def:chain_rule}, we say that $\cX$ {\em admits a chain rule} if for any arc $z \colon \RR_+ \rightarrow \cX$, equality holds $$\dotp{N_{\cX}(z(t)), \dot z(t)} = 0 \qquad\textrm{ for a.e. }t\geq 0.$$ Whenever $\cX$ is Clarke regular or Whitney stratifiable, $\cX$ automatically admits a chain rule. Indeed, the argument is identical to that of  Lemma~\ref{lem:chain_rule_subdiff_reg} and Theorem~\ref{thm:chain_strat}.
 As in the unconstrained case,  Assumption~\ref{ass:bhat}.\ref{ass:b2hat} is true as long as $f$, $g$, and $\cX$  admit a chain rule.

\begin{lem}\label{lem:chaion_rule_gives_assump2}
Consider the optimization problem \eqref{eqn:comp_prob} and suppose that $ f$, $ g$, and $\cX$ admit a chain rule. Let $z\colon\R_+\to\cX$ be any arc satisfying the differential inclusion $$\dot{z}(t)\in G(z(t))\qquad \textrm{for a.e. }t\geq 0.$$ Then equality $\|\dot{z}(t)\|=\dist(0,G(z(t)))$ holds for a.e. $t\geq 0$, and therefore we have the estimate
	\begin{equation}\label{eqn:lyap2}
	\varphi(z(0))-\varphi(z(t))= \int_{0}^t \dist^2\left(0;G(z(\tau))\right)\,d\tau,\qquad \forall t\geq 0.
	\end{equation}
	In particular, property \ref{ass:b2hat} of Assumption~\ref{ass:bhat} holds.
\end{lem}

We now arrive at the main result of the section.

\begin{cor}\label{cor:last_cor_semi_omin_sink}
Suppose that Assumption \ref{ass:stand_ass2} holds and that $f$, $g$, and $\cX$ are definable in an o-minimal structure. Let  $\{x_{k}\}_{k\geq 1}$ be the iterates produced by the proximal stochastic subgradient method \eqref{eqn:stoc_prox_subgrad}. 
Then almost surely, every limit point of  the iterates $\{x_{k}\}_{k\geq 1}$ is composite critical for the problem \eqref{eqn:comp_prob} and the function values $\{\varphi(x_k)\}_{k\geq 1}$ converge.
\end{cor}

\subsection{Comments on boundedness}

Thus far, all of our results have assumed that the subgradient iterates $\{x_k\}$ satisfy $\sup_{k \geq 1} \|x_k\| < \infty$ almost surely. One may enforce this assumption in several ways, most easily by assuming the constraint set $\cX$ is bounded. Beyond boundedness of $\cX$, proper choice of regularizer $g$ may also ensure boundedness of $\{x_k\}$. Indeed, this observation was already made by Duchi-Ruan \cite[Lemma 3.15]{duchi_ruan}. Following their work, let us isolate the following assumption.

\begin{assumption}[Regularizers that induce boundedness]\label{ass:bounded_regularizer}
~

\begin{enumerate}
\item $g$ is convex and $\beta$-coercive, meaning $\lim_{k \rightarrow \infty} g(x)/\|x\|^\beta = \infty.$
\item There exists $\lambda \in (0, 1]$ such that $g(x) \geq g(\lambda x)$ for $x$ with sufficiently large norm.
\end{enumerate}
\end{assumption}

A natural regularizer satisfying this assumption is $\|x\|^{\beta + \epsilon}$ for any $\epsilon > 0$. The following theorem, whose proof is identical to that of 
\cite[Lemma 3.15]{duchi_ruan}, shows that with Assumption~\ref{ass:bounded_regularizer} in place, the stochastic proximal subgradient methods produces bounded iterates. 
\begin{thm}[Boundedness of iterates under coercivity]\label{thm:boundedness_reg}
Suppose that Assumption~\ref{ass:bounded_regularizer} holds and that $\cX = \RR^d$. In addition, suppose there exists $L > 0$ and $\nu < \beta - 1$ such that $\|\zeta(x, \omega)\| \leq L(1 + \|x\|^\nu)$ for all $x \in \RR^d$ and $\omega \in \Omega$. Then $\sup_{k \geq 1} \|x_k\| < \infty$ almost surely.
\end{thm}
We note that in the special (deterministic) case that $\zeta(x, \omega) \in \partial f(x)$ for all $\omega$, the assumption on $\zeta(x, \omega)$ reduces to $\sup_{\zeta \in \partial f(x)}\|\zeta\| \leq L(1+\|x\|^\nu)$, which stipulates that $g$ grows more quickly than $f$.

\appendix
\section{Proofs for the proximal extension}\label{app}
In this section, we follow the notation of Section~\ref{sec:prox_ext}. Namely, we let $\zeta \colon \RR^d \times \Omega \rightarrow \RR^d$ be the stochastic subgradient oracle and $T_{(\cdot)}(\cdot) : (0, \infty) \times \RR^d \rightarrow \RR^d$ the proximal selection.
Throughout, we let $x_k$ and $\omega_k$ be generated by the proximal stochastic subgradient method \eqref{eqn:stoc_prox_subgrad} and suppose that Assumption~\ref{ass:stand_ass2} holds. 
Let $\cF_k :=\sigma(x_j,\omega_{j-1}: j\leq k)$ be the sigma algebra generated by the history of the algorithm.

Let us now formally define the normal cone constructions of variational analysis. For any point $x\in \cX$, the {\em proximal normal cone} to $\cX$ at $x$ is the set 
$$N^P_{\cX}(x):=\{\lambda v\in\R^d: x\in\proj_{\cX}(x+v),\lambda\geq 0\},$$
where $\proj_{\cX}(\cdot)$ denotes the nearest point map to $\cX$.
The {\em limiting normal cone} to $\cX$ at $x$, denoted  $N^L_\cX(x)$, consists of all vector $v\in \R^d$ such that there exist sequences $x_i\in {\cX}$ and $v_i\in N^P_{\cX}(x_i)$ satisfying $(x_i,v_i)\to (x,v)$.
The {\em Clarke normal cone} to $\cX$ at $x$ is then simply
$$N_{\cX}(x):=\cl \conv N^L_\cX(x).$$

\subsection{Auxiliary lemmas}
In this subsection, we record a few auxiliary lemmas to be used in the sequel.
\begin{lem}\label{lem:prox}
There exists a function $L \colon \RR^d \rightarrow \RR_+$, which is bounded on bounded sets, such that for any $x, v \in \RR^d$, and $\alpha > 0$, we have 
$$
\alpha^{-1} \|x - x_+\| \leq 2\cdot L(x)+ 2\cdot \|v\|,
$$
where we set $x_+ := T_{\alpha}(x - \alpha v)$. 
\end{lem}
\begin{proof}
Let $L(\cdot)$ be the function from property \ref{ass:stand_ass2}.\ref{it:mild2}.
From the definition of the proximal map, we deduce
\begin{align*}
& \tfrac{1}{2\alpha } \|x_+ - x\|^2\leq  g(x) - g(x_+) -\dotp{v, x_{+} - x}\leq L(x)\cdot \|x_+ - x\|+ \|v\|\cdot \|x_{+} - x\|.
\end{align*}
Dividing both sides by $\tfrac{1}{2}\|x_+ - x\|$ yields the result.
\end{proof}

\begin{lem}\label{lem:bounded_noise}
Let $\{z_k\}_{k\geq 1}$ be a bounded sequence in $\R^d$ and let $\{\beta_k\}_{k\geq 1}$ be a nonnegative sequence satisfying
$
\sum_{k=1}^\infty \beta_k^2 < \infty.
$
Then almost surely over $\omega\sim P$, we have $\beta_k \zeta(z_k, \omega)\to 0$.
\end{lem}
\begin{proof}
Notice that because $\{z_k\}_{k\geq 1}$ is bounded, it follows that $\{p(z_k)\}$ is bounded. Now consider the random variable $X_k = \beta_k^2 \|\zeta(z_k, \cdot)\|^2$. Due to the estimate 
$$
\sum_{k=1}^\infty \EE\left[X_k\right] \leq \sum_{k=1}^\infty \beta_k^2p(z_k) <\infty,
$$
standard results in measure theory (e.g.,~\cite[Exercise 1.5.5]{tao2011introduction}) imply that $X_k \rightarrow 0$ almost surely.
\end{proof}

\begin{lem}\label{lem:boundedness_2}
Almost surely, we have $\alpha_k\|\zeta(x_k, \omega_k)\| \to 0$ as $k \rightarrow \infty$. 
\end{lem}
\begin{proof}
From the variance bound, $\EE\left[ \|X - \EE\left[X\right]\|^2\right] \leq \EE\left[\|X\|^2\right]$, and Assumption~\ref{ass:stand_ass2}, we have $$\EE\left[\|\zeta(x_k, \omega_{k}) - \EE\left[\zeta(x_k, \omega_{k}) \mid  \cF_k \right]\|^2\mid  \cF_k \right] \leq \EE\left[\|\zeta(x_k, \omega_{k})\|^2\mid  \cF_k \right] \leq p(x_k).$$
Therefore, the following infinite sum is a.s. finite:
		\begin{align*}
		 \sum_{i=1}^\infty \alpha_i^2 \EE\left[\|\zeta(x_i, \omega_{i}) - \EE\left[\zeta(x_i, \omega_{i})\mid  \cF_i \right]\|^2\mid  \cF_i \right]\leq  \sum_{i=1}^\infty \alpha_i^2 p(x_i) < \infty. 
		\end{align*}	
	Define the $L^2$ martingale $X_k = \sum_{i=1}^k \alpha_i (\zeta(x_i, \omega_i) - \EE\left[\zeta(x_i, \omega_i) \mid \cF_i\right])$. Thus, the limit $\dotp{X}_{\infty}$ of the predictable compensator 
	$$
	\dotp{X}_k := \sum_{i=1}^k \alpha_i^2 \EE\left[\|\zeta(x_i, \omega_{i}) - \EE\left[\zeta(x_i, \omega_{i})\mid  \cF_i \right]\|^2\mid  \cF_i \right],
	$$
	exists. Applying \cite[Theorem 5.3.33(a)]{dembo2016probability}, we deduce that almost surely $X_k$ converges to a finite limit, which directly implies $\alpha_k\|\zeta(x_k, \omega_{k}) - \EE\left[\zeta(x_k, \omega_{k})\mid  \cF_k \right]\| \to 0$ almost surely as $k \rightarrow \infty$. Therefore, since $\alpha_k \|\EE\left[\zeta(x_k, \omega_{k})\mid  \cF_k \right] \|\leq \alpha_k\EE\left[ \|\zeta(x_k, \omega_{k})\|\mid  \cF_k \right] \leq \alpha_k \sqrt{p(x_k)} \rightarrow 0$ almost surely as $k \rightarrow 0$, it follows that $\alpha_k\|\zeta(x_k, \omega_{k})\| \rightarrow 0$ almost surely as $k \rightarrow 0$.
\end{proof}

\subsection{Proof Theorem~\ref{thm:main_cray_thmhat}}

 In addition to Assumption~\ref{ass:stand_ass2}, let us now suppose that Assumption~\ref{ass:bhat} holds. 
 Define the set-valued map $G\colon Q\rightrightarrows\R^d$  by
 $G=-\partial f-\partial g-N_{\cX}$.  
 We aim to apply Theorem~\ref{thm:main_cray_thm_general}, which would immediately imply the validity of Theorem~\ref{thm:main_cray_thmhat}. To this end, notice that Assumption~\ref{ass:bhat} is exactly Assumption~\ref{ass:b_general} for our map $G$. Thus we must only verify that Assumption~\ref{ass:duchi_assumption} holds almost surely. Note that properties \ref{ass:duchi_assumption}.\ref{item:iterates_in_Q} and \ref{ass:duchi_assumption}.\ref{item:steps_duchi} hold vacuously. Thus, we must only show that  \ref{ass:duchi_assumption}.\ref{item:gradients_bounded}, \ref{ass:duchi_assumption}.\ref{item:noise}, and~\ref{ass:duchi_assumption}.\ref{item:average_OSC} hold. The argument we present is essentially the same as in \cite[Section 3.2.2]{duchi_ruan} .
 
 For each index $k$, define the set-valued map 
   $$
   G_{k}(x) := -\partial f(x) - \alpha_k^{-1} \cdot \EE_{\omega}\left[x - \alpha_k \zeta(x, \omega) - T_{\alpha_k}(x - \alpha_k\zeta(x, \omega))\right]
   $$
   Note that $G_k$ is a deterministic map, with $k$ only signifying the dependence on the deterministic sequence $\alpha_k$. Define now the noise sequence 
$$
\xi_{k} := \tfrac{1}{\alpha_k}\left[T_{\alpha_k}(x_k - \alpha_k\zeta(x_k, \omega_k)) -x_k\right]-\tfrac{1}{\alpha_k}\left[\EE_{\omega}\left[T_{\alpha_k}(x_k - \alpha_k \zeta(x_k, \omega))-x_k\right])\right].
$$
Let us now write the proximal stochastic subgradient method in the form \eqref{eqn:Euler_rec}.
 \begin{lem}[Recursion relation]
 For all $k \geq 0$, we have
$$
x_{k+1} = x_k + \alpha_k\left[ y_k + \xi_k\right] \qquad \text{for some }y_k \in G_k(x_k).
$$
 \end{lem}
\begin{proof}
Notice that for every index $k\geq 0$, we have
\begin{align*}
 \tfrac{1}{\alpha_k} (x_k - x_{k+1}) &= \tfrac{1}{\alpha_k} \left[x_k -  T_{\alpha_k }(x_k - \alpha_k\zeta(x_k, \omega_k))\right]\\
&= \EE_{\omega}\left[\zeta(x_k, \omega)\right] +\tfrac{1}{\alpha_k} \EE_{\omega}\left[x_k - \alpha_k \zeta(x_k, \omega) - T_{\alpha_k}(x_k - \alpha_k\zeta(x_k, \omega))\right] \\
&\hspace{20pt} + \tfrac{1}{\alpha_k}\left[\EE_{\omega}\left[T_{\alpha_k}(x_k - \alpha_k \zeta(x_k, \omega))\right]-T_{\alpha_k}(x_k - \alpha_k\zeta(x_k, \omega_k))\right]\\
&\in  -G_k(x_k) - \xi_k,
\end{align*}
as desired.
\end{proof}

The following lemma shows that \ref{ass:duchi_assumption}.\ref{item:noise} holds almost surely.
\begin{lem}[Weighted noise sequence] \label{claim:summable}
The limit $\displaystyle\lim_{n \rightarrow \infty} \sum_{i=1}^n \alpha_i \xi_i$ exists almost surely.
\end{lem}
\begin{proof}
We first prove that $\{\alpha_k\xi_k\}$ is an $L_2$ martingale difference sequence, meaning that for all $k$, we have
$$
\EE\left[\alpha_k \xi_k \mid \cF_k\right]  = 0 \qquad\text{and} \qquad  \sum_{k=1}^\infty \alpha_k^2 \EE\left[ \|\xi_k\|^2 \mid \cF_k \right] < \infty.
$$
Clearly, $\xi_k$ has zero mean conditioned on the past, and so we need only focus on the second property.
By the variance bound, $\EE\left[ \|X - \EE\left[X\right]\|^2\right] \leq \EE\left[\|X\|^2\right]$, and Lemma~\ref{lem:prox}, we have
\begin{align*}
\EE\left[ \|\xi_k\|^2 \mid \cF_k \right] &\leq \frac{1}{\alpha_k^2}\EE\left[\left\|T_{\alpha_k}(x_k - \alpha_k\zeta(x_k, \omega_k)) - x_k\right\|^2 \mid \cF_k\right] \\
&\leq 4\cdot L(x_k)^2+  4\cdot\EE\left[ \|\zeta(x_k, \omega_k)\|^2 \mid \cF_k\right].
\end{align*}
Notice that because $\{x_k\}$ is bounded a.s., it follows that $\{L(x_k)\}$ and $\{p(x_k)\}$ are bounded a.s. Therefore, because 
$$
\sum_{k=1}^\infty \alpha_k^2\EE\left[\|\zeta(x_k, \omega_k)\|^2 \mid \cF_k\right] \leq \sum_{k=1}^\infty \alpha_k^2p(x_k) <\infty,
$$
it follows that 
$
\sum_{k=1}^\infty \alpha_k^2\EE\left[\|\xi_k\|^2 \mid \cF_k\right] < \infty,
$ almost surely,
as desired. 

Now, define the $L^2$ martingale $X_k = \sum_{i=1}^k \alpha_i \xi_{i}$. Thus, the limit $\dotp{X}_{\infty}$ of the predictable compensator 
	$$
	\dotp{X}_k := \sum_{i=1}^k \alpha_i^2 \EE\left[\|\xi_i\|^2\mid  \cF_i \right],
	$$
	exists. Applying \cite[Theorem 5.3.33(a)]{dembo2016probability}, we deduce that almost surely $X_k$ converges to a finite limit, which completes the proof of the claim.
\end{proof}

Now we turn our attention to \ref{ass:duchi_assumption}.\ref{item:gradients_bounded}.
\begin{lem}
Almost surely, the sequence $\{y_k\}$ is bounded. 
\end{lem}
\begin{proof}
Because the sequence $\{x_k\}$ is almost surely bounded and $g$ is locally Lipschitz, clearly we have
$$
\sup \left\{ \|v\| : v\in \bigcup_{k \geq 1} \partial f(x_k)\right\} < \infty,
$$
almost surely. Thus, we need only show that 
$$
\sup_{k \geq 1} \left\{\left\|\frac{1}{\alpha_k}  \EE_{\omega}\left[x_k - \alpha_k \zeta(x_k, \omega) - T_{\alpha_k}(x_k - \alpha_k\zeta(x_k, \omega))\right] \right\|\right\} < \infty,
$$
almost surely.  To this end, by the triangle inequality and Lemma~\ref{lem:prox}, we have for any fixed $\omega \in \Omega$ the bound
\begin{align*}
\left\|\tfrac{1}{\alpha_k}\left[x_k - T_{\alpha_k}(x_k - \alpha_k\zeta(x_k, \omega))\right] \right\| \leq 2\cdot L(x_k) + 2\cdot \|\zeta(x_k, \omega)\|
\end{align*}
Therefore, by Jensen's inequality, we have that
\begin{align*}
&\left\|\tfrac{1}{\alpha_k}  \EE_{\omega}\left[x_k - \alpha_k \zeta(x_k, \omega) - T_{\alpha_k}(x_k - \alpha_k\zeta(x_k, \omega))\right] \right\| \\
&\leq 2\cdot L(x_k)  + 3\cdot\EE_{\omega} \left[ \|\zeta(x_k, \omega)\| \right] \\
&\leq 2\cdot L(x_k) + 3\cdot \sqrt{p(x_k)},
\end{align*}
which is almost surely bounded for all $k$. Taking the supremum yields the result.
\end{proof}

As the last step, we verify Item~\ref{ass:duchi_assumption}.\ref{item:average_OSC}.
\begin{lem}
Item~\ref{item:average_OSC} of Assumption~\ref{ass:duchi_assumption} is true.
\end{lem}
\begin{proof}
It will be convenient to prove a more general statement, which is independent of the iterate sequence $\{x_k\}$, and instead only depends on the maps $G_k$. Namely, consider any sequence $\{z_k\} \subseteq \cX$ converging to a point $z\in \cX$ and an arbitrary sequence $w_k^f \in \partial f(z_k)$. Let $\{n_k\}$ be an unbounded increasing sequence of indices. Observe that since $G(z)$ is convex and using Jensen's inequality, we have
\begin{align*}
&\dist\left(\frac{1}{n}\sum_{k=1}^n\left(-w_k^f -  \frac{1}{\alpha_{n_k}} \EE_{\omega} \left[z_k - \alpha_{n_k} \zeta(z_k, \omega) - T_{\alpha_{n_k}}(z_k - \alpha_{n_k}\zeta(z_k, \omega))\right]\right), G(z)\right)\\
& \leq  \frac{1}{n}\sum_{k=1}^n\EE_{\omega}\left[ \dist\left(-w_k^f -  \frac{1}{\alpha_{n_k}} \left[z_k - \alpha_{n_k} \zeta(z_k, \omega) - T_{\alpha_{n_k}}(z_k - \alpha_{n_k}\zeta(z_k, \omega))\right], G(z)\right)\right].
\end{align*}
Our goal is to prove that the right-hand-side tends to zero almost surely, which directly implies validity of  
\ref{ass:duchi_assumption}.\ref{item:average_OSC}

 Our immediate goal is to apply the dominated convergence theorem to each term in the above finite sum to conclude that each term converges to zero. 
 To that end, we must show two properties: for every fixed $\omega$, each term in the sum tends to zero, and that each term is bounded by an integrable function. We now prove both properties.
\begin{claim}\label{subclaim:pointwise}
Almost surely in $\omega\sim P$, we have that
$$
\dist\left(-w_k^f -  \frac{1}{\alpha_{n_k}} \left[z_k - \alpha_{n_k} \zeta(z_k, \omega) - T_{\alpha_{n_k}}(z_k - \alpha_{n_k}\zeta(z_k, \omega))\right], G(z)\right) \rightarrow 0 \quad \text{as $k \rightarrow \infty$}.
$$
\end{claim}
\begin{proof}[Proof of Subclaim~\ref{subclaim:pointwise}]
Optimality conditions \cite[Exercise 10.10]{RW98} of the proximal subproblem imply
$$
 \tfrac{1}{\alpha_{n_k}} \left[z_k - \alpha_{n_k} \zeta(z_k, \omega) - T_{\alpha_{n_k}}(z_k - \alpha_{n_k}\zeta(z_k, \omega))\right]  =  w^g_k(\omega) + w^\cX_k(\omega),
$$
for some $w^g_k(\omega)\in \partial g(T_{\alpha_{n_k}}(z_k -  \alpha_{n_k} \zeta(z_k, \omega)))$ and $w^\cX_k(\omega)\in N_{\cX}^{L}(T_{\alpha_{n_k}}(z_k -  \alpha_{n_k} \zeta(z_k, \omega)))$, and where $N_{\cX}^{L}$ denotes the limiting normal cone. Observe that by continuity and the fact that $\sum_{k=1}^\infty \alpha_{n_k}^2 < \infty$ and $\alpha_{n_k} \zeta(z_k, \omega) \rightarrow 0$ as $k \rightarrow \infty$ a.e. (see Lemma~\ref{lem:bounded_noise}), it follows that 
$$
T_{\alpha_{n_k}}(z_k -  \alpha_{n_k} \zeta(z_k, \omega)) \rightarrow z.
$$
 Indeed, setting $z_k^+ = T_{\alpha_{n_k}}(z_k -  \alpha_{n_k} \zeta(z_k, \omega))$, we have that by Lemma~\ref{lem:prox},
$$
\|z_k - z_{k}^+\| \leq 2\alpha_{n_k} L(z_k) + 2\alpha_{n_k} \| \zeta(z_k, \omega)\| \rightarrow 0 \qquad \text{as $k \rightarrow \infty$},
$$
which implies that $\lim_{k \rightarrow \infty} z_k^+ = \lim_{k \rightarrow \infty} z_k = z$. 

We furthermore deduce that $w^\cX_k(\omega)$ and $w_k^g(\omega)$ are bounded almost surely. Indeed, $w_k^g(\omega)$ is bounded since $g$ is locally Lipschitz and $z_k^+$ are bounded. Moreover,  Lemma~\ref{lem:prox} implies
$$
\|w_k^g(\omega) + w_k^\cX(\omega)\|= \left\|\tfrac{1}{\alpha_{n_k}} \left[z_k - \alpha_{n_k} \zeta(z_k, \omega) - z_k^+\right] \right\|\leq 2\cdot L(z_k) + 3\cdot \sup_{k \geq 1} \|\zeta(z_k,\omega)\|.
$$ Observe that the right hand-side is a.s. bounded by item~\ref{it:loc_integrability} of Assumption~\ref{ass:stand_ass2}. Thus, since $w_k^g(\omega) + w_k^\cX(\omega)$ and $w_k^g(\omega)$ are a.s. bounded, it follows that $w_k^\cX(\omega)$ must also be a.s. bounded, as desired.

Appealing to outer semicontinuity of $\partial f, \partial g,$ and $N^{L}_{\cX}$ (e.g. \cite[Propostion 6.6]{RW98}), the inclusion $N^{L}_{\cX}\subset N_{\cX}$, and the boundedness of $\{w^f_k\}, \{w^g_k(\omega)\},$ and $\{w^\cX_k(\omega)\}$, it follows that 
$$
\dist( w^f_k, \partial f(z)) \rightarrow 0; \qquad \dist( w^g_k(\omega), \partial g(z)) \rightarrow 0;\qquad \dist(w^\cX_k(\omega), N_{\cX}(z)) \rightarrow 0,
$$
as $ k \rightarrow \infty$. Consequently, almost surely we have that 
\begin{align*}
&\dist\left(-w_k^f -  \tfrac{1}{\alpha_{n_k}} \left[z_k - \alpha_{n_k} \zeta(z_k, \omega) - T_{\alpha_{n_k}}(z_k - \alpha_{n_k}\zeta(z_k, \omega))\right], G(z)\right)\\
&\leq \dist( w^f_k, \partial f(z)) + \dist( w^g_k(\omega), \partial g(z)) + \dist(w^\cX_k(\omega), N_{\cX}(z)) \rightarrow 0,
\end{align*}
as desired.
\end{proof}

\begin{claim}\label{subclaim:integrable}
Let $L_f := \sup_{k \geq 1}\dist(0, \partial f(z_k))$ and $ L_g := \sup_{k \geq 1}L(z_{k})$. Then for all $k\geq 0$, the functions
\begin{align*}
&\dist\left(-w_k^f -  \tfrac{1}{\alpha_{n_k}} \left[z_k - \alpha_{n_k} \zeta(z_k, \omega) - T_{\alpha_{n_k}}(z_k - \alpha_{n_k}\zeta(z_k, \omega))\right], G(z)\right) 
\end{align*}
 are uniformly dominated by an integrable function in $\omega$.
\end{claim}
\begin{proof}[Proof of Subclaim~\ref{subclaim:integrable}]
For each $k$, Lemma~\ref{lem:prox} implies the bound 
$$
 \left\| \tfrac{1}{\alpha_{n_k}} \left[z_k - \alpha_{n_k} \zeta(z_k, \omega) - T_{\alpha_{n_k}}(z_k - \alpha_{n_k}\zeta(z_k, \omega))\right] \right\| \leq 2L_g + 3\cdot \|\zeta(z_k, \omega)\|.
$$
Consequently, we have 
\begin{align*}
&\dist\left(-w_k^f -  \tfrac{1}{\alpha_{n_k}} \left[z_k - \alpha_{n_k} \zeta(z_k, \omega) - T_{\alpha_{n_k} r}(z_k - \alpha_{n_k}\zeta(z_k, \omega))\right], G(z)\right) \\
 &\leq L_f + 2L_g + 3\cdot \|\zeta(z_k, \omega)\| + \dist(0, G(z))\\
 &\leq  L_f + 2L_g + 3\cdot \sup_{k \geq 1} \|\zeta(z_k,\omega)\| +\dist(0, \partial f(z) + \partial g(z)),
\end{align*}
which is integrable by Item~\ref{it:loc_integrability} of Assumption~\ref{ass:stand_ass2}.
\end{proof}

Applying the dominated convergence theorem, it follows that
$$
\EE_{\omega}\left[ \dist\left(-w_k^f -  \tfrac{1}{\alpha_{n_k}} \left[z_k - \alpha_{n_k} \zeta(z_k, \omega) - T_{\alpha_{n_k}}(z_k - \alpha_{n_k}\zeta(z_k, \omega))\right], G(z)\right)\right] \rightarrow 0
$$
as $k \rightarrow \infty$. Notice the simple fact that for any real sequence $b_k \to 0 $, it must be that $\frac{1}{n} \sum_{k=1}^{n}b_k \rightarrow 0$ as $n \rightarrow \infty$. Consequently 
\begin{align*}
&\dist\left(\frac{1}{n}\sum_{k=1}^n\left(-w_k^f - \tfrac{1}{\alpha_{n_k}} \EE_{\omega} \left[z_k - \alpha_{n_k} \zeta(z_k, \omega) - T_{\alpha_{n_k}}(z_k - \alpha_{n_k}\zeta(z_k, \omega))\right]\right), G(z)\right)\\
& \leq \frac{1}{n}\sum_{k=1}^n\EE_{\omega}\left[ \dist\left(-w_k^f -  \tfrac{1}{\alpha_{n_k}} \left[z_k - \alpha_{n_k} \zeta(z_k, \omega) - T_{\alpha_{n_k}}(z_k - \alpha_{n_k}\zeta(z_k, \omega))\right], G(z)\right)\right] \rightarrow 0
\end{align*}
as $n \rightarrow \infty$. This completes the proof.
\end{proof}
We have now verified all parts of Theorem~\ref{thm:duchi_arzela}. Therefore, the proof is complete. 

\subsection{Verifying Assumption~\ref{ass:bhat} for composite problems}
\begin{proof}[Proof of Lemma~\ref{lem:chaion_rule_gives_assump2}]
The argument is nearly identical to that of Lemma~\ref{lem:chaion_rule_gives_assump}, with one additional subtlety that $G$ is not necessarily outer-semicontinuous. Let $z\colon\R^d\to\cX$ be an arc. Since $f$, $g$, and $\cX$ admit a chain rule, we deduce 
\begin{equation*}\label{eqn:rand_stuff1}
(f\circ z)'(t)=\langle \partial f(z(t)),\dot{z}(t)\rangle \quad (g\circ z)'(t)=\langle \partial g(z(t)),\dot{z}(t)\rangle,\quad  \textrm{and}\quad 0=\langle N_{\cX}(z(t)),\dot z(t)\rangle,
\end{equation*}
for a.e. $t\geq 0$.
Adding the three equations yields
$$(\varphi\circ z)'(t)=-\langle G(z(t)),\dot{z}(t)\rangle \quad\textrm{ for a.e. }t\geq 0.$$
Suppose now that $z(\cdot)$ satisfies 
$\dot{z}(t)\in -G(z(t))$ for a.e. $t\geq 0$.
Then the same linear algebraic argument as in Lemma~\ref{lem:chaion_rule_gives_assump} yields the equality $\|\dot{z}(t)\|= \dist(0;G(z(t)))$ for a.e. $t\geq 0$ and consequently the equation \eqref{eqn:lyap2}. 

To complete the proof, we must only show that property \ref{ass:b2hat} of Assumption~\ref{ass:bhat} holds. To this end, suppose that $z(0)$ is not composite critical and let $T>0$ be arbitrary.
Appealing to \eqref{eqn:lyap2}, clearly $\sup_{t\in [0,T]} \varphi(z(t))\leq \varphi(z(0))$. Thus we must only argue $\varphi(z(T))<\varphi(z(0))$. According to \eqref{eqn:lyap2}, if this were not the case, then we would deduce $\dist(0;G(z(t)))=0$ for a.e. $t\in [0,T]$. Appealing to the equality $\|\dot z\|=\dist(0;G(z(t)))$, we therefore conclude $\|\dot{z}\|=0$ for a.e. $t\in [0,T]$. Since $z(\cdot)$ is absolutely continuous, it must therefore be constant $z(\cdot)\equiv z(0)$, but this is a contradiction since $0\notin G(z(0))$. Thus property \ref{ass:b2hat} of Assumption~\ref{ass:bhat} holds, as claimed.
\end{proof}

\begin{proof}[Proof of Corollary~\ref{cor:last_cor_semi_omin_sink}]
The result follows immediately from Lemma~\ref{thm:main_cray_thmhat}, once we show that Assumption~\ref{ass:bhat} holds. Since $f$ and $g$ are definable in an o-minimal structure, Theorem~\ref{thm:chain_strat} implies that $f$ and $ g$ admit the chain rule. The same argument as in Theorem~\ref{thm:chain_strat} moreover implies $\cX$ admits the chain rule as well. Therefore, Lemma~\ref{lem:chaion_rule_gives_assump2} guarantees that the descent property of Assumption~\ref{ass:bhat} holds. Thus we must only argue the weak Sard property of Assumption~\ref{ass:bhat}. To this end, since $f$, $g$, and $\cX$ are definable in an o-minimial structure, there exist Whitney $C^{d}$-stratifications $\mathcal{A}_f$, $\mathcal{A}_g$, and $\mathcal{A}_{\cX}$  of $\gph f$, $\gph g$, and $\cX$, respectively. Let $\Pi\mathcal{A}_f$ and $\Pi\mathcal{A}_g$ be the Whitney stratifications of $\R^d$ obtained by applying the coordinate projection $(x,r)\mapsto x$ to each stratum in $\mathcal{A}_f$ and $\mathcal{A}_g$. Appealing to \cite[Theorem 4.8]{Dries-Miller96}, we obtain a Whitney $C^d$-stratification $\mathcal{A}$ of $\R^d$ that is compatible with $(\Pi\mathcal{A}_f,\Pi\mathcal{A}_g, \mathcal{A}_\cX)$. That is, for every strata $M\in \mathcal{A}$ and $L\in \Pi\mathcal{A}_f\cup \Pi\mathcal{A}_g\cup \mathcal{A}_\cX$, either $M\cap L=\emptyset$ or $M\subseteq L$. 

Consider an arbitrary stratum $M\in\mathcal{A}$ intersecting $\cX$ (and therefore contained in $\cX$) and a point $x\in M$. Consider now the (unique) strata $M_f\in \Pi\mathcal{A}_{f}$, $M_g\in \Pi\mathcal{A}_{g}$, and $M_{\cX}\in \mathcal{A}_\cX$ containing $x$. Let $\widehat{f}$ and $\widehat{g}$ be $C^d$-smooth functions agreeing with $f$ and $g$ on a neighborhood of $x$ in $ M_f$ and $ M_g$, respectively. Appealing to \eqref{eqn:subdiff_inclusion}, we conclude 
$$\partial f(x)\subset \nabla \widehat f(x)+N_{ M_f}(x) \qquad \textrm{and}\qquad \partial g(x)\subset \nabla \widehat g(x)+N_{M_g}(x).$$
The Whitney condition in turn directly implies
$N_{\cX}(x)\subset N_{M_{\cX}}(x).$ Hence summing yields 
\begin{align*}
\partial f(x)+\partial g(x)+N_{\cX}(x)&\subset \nabla (\widehat f+\widehat g)(x)+N_{ M_f}(x)+N_{ M_g}(x)+N_{M_{\cX}}(x)\\
&\subset \nabla (\widehat f+\widehat g)(x)+N_{M}(x),
\end{align*}
where the last inclusion follows from the compatibility $M\subset M_f$ and $M\subset M_g$. 
Notice that $\widehat f+\widehat g$ agrees with $f+g$ on a neighborhood of $x$ in $M$. Hence if the inclusion, $0\in \partial f(x)+\partial g(x)+N_{\cX}(x)$, holds it must be that $x$ is a critical point of the $C^d$-smooth function $f+g$ restricted to $M$, in the classical sense. Applying the standard Sard's theorem to each manifold $M$, the result follows. 
\end{proof}

\let\oldbibliography\thebibliography
\renewcommand{\thebibliography}[1]{%
	\oldbibliography{#1}%
	\setlength{\itemsep}{-2pt}%
}

			\bibliographystyle{plain}
	\bibliography{bibliography}
\end{document}